\documentclass[11pt,reqno]{amsart}
\usepackage{amsmath,amssymb,amsthm,mathrsfs}
\usepackage{tikz}
\usepackage{enumitem}

\usetikzlibrary{arrows.meta}
\usetikzlibrary{decorations.pathreplacing}

\usepackage[margin=1.1in]{geometry}
\usepackage[hidelinks]{hyperref}
\hypersetup{
  pdftitle={Height of minimal surfaces: Sharp thresholds},
  pdfauthor={Tobias Holck Colding and William P. Minicozzi II},
  pdfsubject={Minimal surfaces and proper holomorphic curves in cones}
}

\newtheorem{theorem}{Theorem}[section]
\newtheorem{proposition}[theorem]{Proposition}
\newtheorem{lemma}[theorem]{Lemma}
\newtheorem{corollary}[theorem]{Corollary}
\theoremstyle{definition}
\newtheorem{remark}[theorem]{Remark}

\newtheorem{definition}[theorem]{Definition}

\newcommand{\cA}{\mathcal{A}}
\newcommand{\R}{\mathbb{R}}
\newcommand{\C}{\mathbb{C}}
\newcommand{\CP}{\mathbb{CP}}
\newcommand{\N}{\mathbb{N}}
\newcommand{\D}{\mathbb{D}}

\newcommand{\cH}{\mathcal{H}}
\DeclareMathOperator{\Rot}{Rot}
\DeclareMathOperator{\spt}{spt}
\newcommand{\Pin}{\Pi}
\newcommand{\Pout}{\Pi^{\perp}}
\DeclareMathOperator{\Area}{Area}
\newcommand{\Vol}{\operatorname{Vol}}

\DeclareMathOperator{\e}{e}
\newcommand{\Tr}{{\text{Tr}}}
\DeclareMathOperator{\Real}{Re}

\newcommand{\ddc}{d\hspace{0.5pt}d^{c}}

\newcommand{\Pbot}{P^{\perp}}
\newcommand{\Cone}[1]{C_{#1}(\Pi)}
\newcommand{\Slab}[2]{S_{#1,#2}(\Pi)}

\DeclareMathOperator{\arcosh}{arcosh}

\begin{document}

\title[Height of minimal surfaces]{Height of minimal surfaces: sharp thresholds}

\author{Tobias Holck Colding}
\address{Department of Mathematics, MIT,
77 Massachusetts Avenue, Cambridge, MA 02139}
\email{colding@math.mit.edu and minicozz@math.mit.edu}

\author{William P. Minicozzi II}

\thanks{The authors were partially supported by NSF DMS Grants 2405393 and 2304684.}

\begin{abstract}
The confinement theorem of \cite{CM1} gives Euclidean volume growth for a proper minimal submanifold whose height over a plane is sublinear.   Such volume growth places strong restrictions on the geometry.  This leaves linear height as the natural borderline and raises the question of which, if any, linearly growing cones still force Euclidean  growth.

Here we determine this borderline sharply for proper holomorphic curves in $\C^2$.  There are two sharp thresholds for the aperture depending on the K\"ahler angle $\theta$ of the two-plane $\Pin$: $\tan(\theta/2)$ and $\cot(\theta/2)$. No proper holomorphic curve exists below the first, every such curve is algebraic below the second, and superquadratic proper holomorphic disks exist above the second.   
These conclusions extend to holomorphic submanifolds of arbitrary dimension and codimension.  

For stationary integral varifolds there are again three regimes: sublinear height
forces Euclidean growth, narrow cones of aperture $\varepsilon$ force polynomial
growth with exponent $n+O(\varepsilon^2)$, and wide cones allow arbitrarily fast
growth.  The order $\varepsilon^2$ is sharp: already in $\R^3$, for every small
$\varepsilon>0$ there are examples in a cone of aperture $\varepsilon$ whose
volume grows at least like $r^{2+\frac45\varepsilon^2}$.

\end{abstract}

\maketitle
\enlargethispage{2pt}

\setcounter{section}{-1}
\section{Introduction}

Solutions of elliptic equations tend to spread out evenly, and confining them
imposes strong restrictions.  Two classical linear examples illustrate this
particularly well.  The first is Liouville's theorem: a positive harmonic
function on $\R^n$ is constant, and the Harnack inequality is its local
quantitative form.  The second is the dichotomy for holomorphic functions: an
entire function of polynomial growth is a polynomial and takes each value
finitely often, while a transcendental one takes every value, with at most one
exception, infinitely often.  Equivalently, an entire function is a polynomial
exactly when its graph in $\C^2$ has Euclidean area growth, which is the
simplest case of the Bishop--Stoll theorem discussed below.

The same principle appears, more subtly, for nonlinear equations.  On a minimal
surface the coordinate functions are harmonic, but for the induced metric, which
is itself unknown.  The Hoffman-Meeks strong halfspace theorem \cite{HM} is a nonlinear
Liouville theorem: a properly immersed minimal surface in $\R^3$ on which the
height $x_3$ is positive must be a plane.  For embedded surfaces, a local
quantitative form, playing the role of the Harnack inequality, is the one-sided
curvature estimate \cite{CMA1,CMA2,CMA3,CMA4}: an embedded minimal disk that lies on one
side of a plane and comes close to it has bounded curvature near the plane.
This estimate is a key tool in the theory of embedded minimal surfaces in
$\R^3$, which is by now quite well understood, 
  \cite{BB1,BB2,CMbook,CMA4,MP,P2}.

Unlike Liouville's theorem, however, halfspace rigidity fails in higher
dimension and codimension: already in codimension one, the higher-dimensional
catenoid and many other examples are nonflat complete properly embedded minimal
hypersurfaces contained in a slab \cite{Blair}.  What survives is quantitative:
by \cite{CM1}, a properly immersed minimal submanifold, of any dimension and
codimension, whose height over a plane grows sublinearly has Euclidean volume
growth.  This growth bound has strong geometric consequences.  At large scales,
compactness and monotonicity imply that every sequence of blow-downs has a
subsequence converging to a minimal cone, thereby producing a tangent cone at
infinity.  In $\R^3$ the consequences are even more rigid: among properly
embedded minimal surfaces of finite genus, quadratic area growth is equivalent
to finite total curvature, and the latter class has long been understood through
the classical theory \cite{JM,Os}.  The first case not covered by \cite{CM1} is
linear height, that is, confinement to a cone around the plane.  For holomorphic
curves in $\C^2$ the two classical examples meet: they are minimal
surfaces,\footnote{Holomorphic submanifolds of $\C^m$ are area-minimizing by
Wirtinger \cite{W,HL}, and thus minimal and stable.} and by Bishop--Stoll
\cite{Bishop,S} they are algebraic exactly when their area growth is Euclidean,
so for them the question becomes which cones force algebraicity.

For proper holomorphic curves in $\R^4=\C^2$ we answer this sharply: there are
two thresholds for the cone aperture, depending on the K\"ahler angle of the
plane.  Below the first, no such curve exists; below the second, every such curve
is algebraic, with quadratic area growth; and above the second, examples with
superquadratic area growth exist.

Holomorphic curves in $\C^2$ are an exception to the failure of halfspace
rigidity in higher codimension: Theorem~2 of \cite{SW} implies that a
nonconstant proper holomorphic curve contained in a real halfspace has affine
complex-line image.  In particular, no proper holomorphic disk can be contained
in a real halfspace.  This conclusion is special to $\C^2$: if $G:\D\to\C^2$ is
any proper holomorphic map from the disk $\D=\{z\in\C:|z|<1\}$, then
\[
   \zeta\longmapsto\bigl(G(\zeta),1+\zeta\bigr)
\]
is a proper holomorphic disk in $\C^3$ contained in the halfspace
$\{\operatorname{Re}z_3>0\}$.  The examples below instead rely essentially on the
fact that the cones $C_\varepsilon(\Pin)\subset\C^2$ are nonconvex.  A new proof
of the halfspace theorem in \cite{CMdistance} makes clear that its rigidity is
closely tied to the parabolicity of the two-dimensional plane, and this
distinction reappears here: the holomorphic curves in the rigidity range are
parabolic.

 For stationary integral varifolds, the weak solutions of the minimal surface
equation, there are again three regimes, now in every dimension and codimension
(Theorem~\ref{t:realsetting}).  Sublinear height forces Euclidean volume growth
\cite{CM1}; in a sufficiently narrow cone of aperture $\varepsilon$ the volume
grows at most like $r^{n+c_n\varepsilon^2}$, which approaches the Euclidean rate
as the cone closes up; and in wide cones the growth can be arbitrarily fast.  Two families of examples show what cannot be improved.  Sums of
scaled catenoids with large multiplicities lie in the cone of aperture
$\max_{t\geq0}t/\cosh t\approx0.66$ and grow arbitrarily fast
(Theorem~\ref{t:catenoid-sum}).  At the other end, and unlike for holomorphic curves, linear
height never forces Euclidean growth: already in $\R^3$, for every small
$\varepsilon>0$ there is a stationary integral $2$-varifold, made of flat pieces
and pieces of catenoids that meet along circles, which lies in a cone of
aperture $\varepsilon$ and has volume growth at least
$r^{2+\frac45\varepsilon^2}$ (Theorem~\ref{t:varifold} and
Figure~\ref{f:varifold3d}).  Thus in the middle regime the growth can be
genuinely faster than Euclidean, and the dependence of the exponent on
$\varepsilon^2$ is sharp.

\subsection{Holomorphic threshold}  The implications of confinement depend on the type of plane. 
Fix a constant orthogonal complex structure $J$ on $\R^4$ and identify $(\R^4,J)$ with $\C^2$.
A two-plane $\Pin$ is \emph{complex} if $J \Pin = \Pin$; $\Pin$ is 
 \emph{Lagrangian} if $J\Pin=\Pin^\perp$ or, equivalently, the K\"ahler form
$\omega(x,y)=\langle Jx,y\rangle$ vanishes on $\Pin$. 
 The \emph{K\"ahler
angle} $\theta=\theta(\Pin)\in[0,\pi/2]$ is determined by
$\cos\theta=|\omega(e_1,e_2)|$ for an orthonormal basis of $\Pin$.  Thus $\theta=0$  when $\Pin$ 
is a complex line, while $\theta=\pi/2$  for a Lagrangian plane.

 Let $P$ and $P^\perp$ denote the
orthogonal projections onto $\Pin$ and $\Pin^\perp$, respectively.  For $\varepsilon,K\geq0$, set (see Figure \ref{f:bounded-error-cone})
\[
 S_{\varepsilon,K}(\Pin)
 =\{x\in\R^4:|P^\perp x|\leq\varepsilon|Px|+K\},
 \qquad C_\varepsilon(\Pin)=S_{\varepsilon,0}(\Pin).
\]
We will refer to $\varepsilon$ as the aperture of the cone.
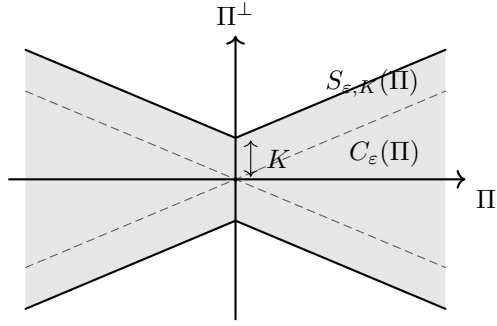
\begin{figure}[ht]
\centering
\begin{tikzpicture}[
   x=1.15cm,
   y=1.15cm,
   scale=.7,
   every node/.style={font=\small},
   line cap=round,
   line join=round
]
   \coordinate (O) at (0,0);
   \def\xmax{3.45}
   \def\slope{0.42}
   \def\offset{0.68}

   \fill[black!10]
      (-\xmax,{\slope*\xmax+\offset})
      --(0,\offset)
      --(\xmax,{\slope*\xmax+\offset})
      --(\xmax,{-\slope*\xmax-\offset})
      --(0,-\offset)
      --(-\xmax,{-\slope*\xmax-\offset})
      --cycle;

   \draw[->,thick] (-3.72,0)--(3.78,0)
      node[below right] {$\Pin$};
   \draw[->,thick] (0,-2.30)--(0,2.36)
      node[above] {$\Pin^\perp$};

   \draw[thick]
      (-\xmax,{\slope*\xmax+\offset})
      --(0,\offset)
      --(\xmax,{\slope*\xmax+\offset});
   \draw[thick]
      (-\xmax,{-\slope*\xmax-\offset})
      --(0,-\offset)
      --(\xmax,{-\slope*\xmax-\offset});

   \draw[densely dashed,black!65]
      (-\xmax,{\slope*\xmax})--(O)--(\xmax,{\slope*\xmax});
   \draw[densely dashed,black!65]
      (-\xmax,{-\slope*\xmax})--(O)--(\xmax,{-\slope*\xmax});

   \draw[<->,thin] (.25,.04)--(.25,{\offset-.04})
      node[midway,right=2pt] {$K$};

   \node at (2.25,1.58) {$S_{\varepsilon,K}(\Pin)$};
   \node at (2.45,.43) {$C_\varepsilon(\Pin)$};
   \fill (O) circle (1.2pt);
\end{tikzpicture}
\caption{The dashed lines bound the cone
$C_\varepsilon(\Pin)$ with aperture $\varepsilon$.  The cone is contained in $S_{\varepsilon,K}(\Pin)$ where the constant $K$ permits a bounded
transverse error.}
\label{f:bounded-error-cone}
\end{figure}
 For a map
$F:\Sigma\to\R^n$, write
\[
 A_F(R)=\Area\bigl(F^{-1}(B_R(0))\bigr),
\]
counted
with multiplicity; when the map is clear, we abbreviate this to $A(R)$.  Recall that an \emph{affine
algebraic curve} in $\C^2$ is the zero set of a nonconstant
polynomial in two complex variables.

\vskip1mm

The next theorem gives three regimes for nonconstant proper
holomorphic curves.

\begin{figure}[ht]
\centering
\begin{tikzpicture}[
    line width=0.6pt,
    >={Stealth[length=2.6mm,width=2mm]},
    every node/.style={font=\small}
  ]

  \def\angHi{35}      
  \def\angLo{19}      
  \def\Lhi{7.4}       
  \def\Llo{6.6}       

  \fill[gray!12] (0,0) -- (\angHi:\Lhi) -- (0, {\Lhi*sin(\angHi)}) -- cycle;      
  \fill[gray!25] (0,0) -- (\angHi:\Lhi) -- (\angLo:\Llo) -- cycle;             
  \fill[gray!6]  (0,0) -- (\angLo:\Llo) -- ({\Llo*cos(\angLo)},0) -- cycle;      
  
  \draw[->] (0,0) -- (8.0,0) node[pos=0.88, below=5pt] {$\Pi$};
  \draw[->] (0,0) -- (0,5.4)  node[pos=0.92, left=5pt]  {$\Pi^{\perp}$};
  
  \draw (0,0) -- (\angHi:\Lhi) node[right, xshift=2pt] {$\cot\dfrac{\theta}{2}$};
  \draw (0,0) -- (\angLo:\Llo) node[right, xshift=2pt] {$\tan\dfrac{\theta}{2}$};

  \path (0,0) -- (\angHi:\Lhi)
        node[sloped, above, pos=0.52] {Rapid growth examples};
  \path (0,0) -- (\angLo:\Llo)
        node[sloped, above, pos=0.70] {Algebraic, quadratic growth};
  \path (0,0) -- (6.2,0)
        node[above, pos=0.62] {Nonexistence};

\end{tikzpicture}
\caption{The three \emph{holomorphic} aperture regimes.  The thresholds are determined by the K\"ahler angle $\theta$ of the two-plane $\Pi$.}
\label{fig:kahler-angle-regions}
\end{figure}
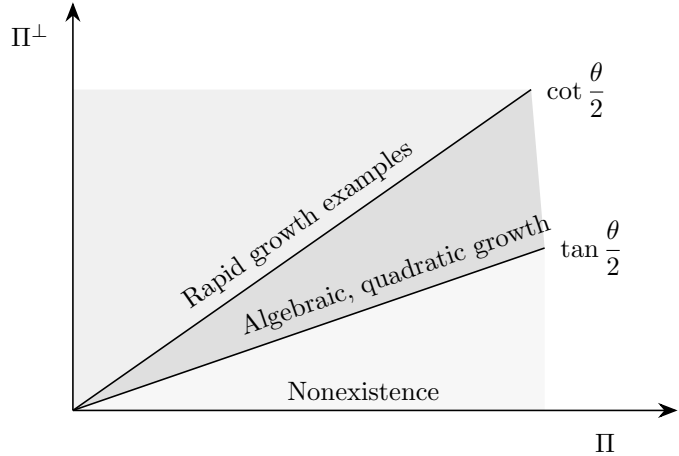

\begin{theorem}[Holomorphic thresholds]	\label{t:angleversion}
We have the three regimes (see Figure \ref{fig:kahler-angle-regions}):
\begin{enumerate}
\item[(i)] If $\varepsilon<\tan(\theta/2)$, then no holomorphic curve is contained in
$S_{\varepsilon,K}(\Pin)$ for any $K<\infty$.
\item[(ii)] If $\tan(\theta/2)\leq\varepsilon<\cot(\theta/2)$, then such curves exist, but every
one has affine algebraic image and quadratic area growth, and its domain is parabolic.
\item[(iii)] If $\varepsilon>\cot(\theta/2)$, then there are complete proper holomorphic
immersions of the disk into $C_\varepsilon(\Pin)$ whose area grows arbitrarily fast.
\end{enumerate}
\end{theorem}

When $\theta=0$ and the plane is complex, the upper threshold is infinite and the third regime is absent.
When $\theta = \frac{\pi}{2}$, then $\cot \frac{\theta}{2} = \tan \frac{\theta}{2}$ and the two thresholds coincide.

The same picture extends to arbitrary complex dimension and codimension.  Suppose that $X$ has complex
dimension $q$ and $F:X\to\C^N$ is a proper holomorphic map of generic rank $q$.  Let
$\Pin\subset\C^N$ be a real $2q$-plane with K\"ahler angles
$0\leq\theta_1\leq\cdots\leq\theta_q\leq\pi/2$.  Nonexistence holds below
$\tan(\theta_1/2)$, while confinement below $\cot(\theta_q/2)$ forces the image to be affine
algebraic with Euclidean volume growth.  When $\theta_1>0$, super-Euclidean 
holomorphic immersions of $\D^q$ exist. 

The examples from part  (iii) in Theorem \ref{t:angleversion}
 lie at the borderline of the confinement results of \cite{CM1}.
Sublinearly growing height forces Euclidean volume growth in every dimension and codimension by 
\cite{CM1}, but (iii) in Theorem \ref{t:angleversion} shows that linear height need not do so.  
The broader dichotomy between confinement and volume growth is discussed in \cite{CM2}. 
  Related Liouville and Bernstein rigidity for
properly immersed minimal disks with quadratic area growth is proved in \cite{CM3}.  Proper
minimal surfaces with exponential or Gaussian area growth, without the present cone confinement,
are constructed in \cite{CMM}.

There is an additional reason why Euclidean volume growth is especially natural
in the holomorphic setting: it is closely tied to algebraicity.  Complex
analytic subvarieties are   area minimizing, and
the bridge between analysis and algebra is anchored by two foundational
results.  Chow's theorem of 1949 \cite{Chow} states that every closed complex
analytic subset of complex projective space $\CP^n$ is algebraic, and hence is
globally defined by a system of homogeneous polynomials.  In affine space
$\C^n$, however, compactness is lost, and closed complex analytic varieties can
exhibit highly transcendental behavior at infinity, as illustrated by the
graph $\{w=e^z\}\subset\C^2$.  The Bishop--Stoll theorem of 1964
\cite{Bishop,S} identifies the precise analytic replacement for compactness: a
closed complex analytic variety in $\C^n$ is algebraic if and only if it has
Euclidean volume growth.

\subsection{Holomorphic doubling}
The motivating narrow-cone question is whether a complete properly immersed minimal surface
contained in $S_{\varepsilon,K}(\Pin)$ for some $K<\infty$ and sufficiently small $\varepsilon$
must have quadratic area growth.  Part (ii) in Theorem \ref{t:angleversion} answers this sharply for
holomorphic curves, and in fact gives algebraicity.

The mechanism behind the quadratic-growth regime parallels the volume estimates in \cite{CM1}. The key input is a doubling theorem that converts geometric confinement into quantitative control: if, at a given scale, a holomorphic curve is confined to a cone whose aperture lies below the superquadratic threshold, then its volume at twice that scale is bounded by a controlled multiple of its volume at the original scale. We state this theorem next.  

Let $F:\Sigma\to(\R^4,J)=\C^2$ be a   holomorphic map from a  
Riemann surface $\Sigma$ and  $\Pin$  a two-plane with K\"ahler angle $\theta$. 

\begin{theorem}[Holomorphic doubling]	\label{t:doubling}
Given $\varepsilon  < \cot \frac{\theta}{2}$ and $K$, there exist $R_1$, $\Omega$, $\omega$ and $C$ so that if  $R > 4\, R_1$, $F^{-1} (B_{\Omega \, R})$ is compact and disjoint from $\partial \Sigma$, and
$$
	B_{\Omega \, R} \cap F(\Sigma) \subset S_{\varepsilon , K} (\Pin) \, , 
$$
then $A_F (R) \le C \, R^2 \, A_F (\omega \, R_1)$.
\end{theorem}

\subsection{Minimal surface thresholds}

The techniques here are flexible and apply also to weak solutions of the minimal surface equation known as stationary varifolds. 
For stationary varifolds of arbitrary dimension and codimension, there are again three regimes depending on the growth of the height.   In this case,  $\Pin$ is an $n$-plane in $\R^{n+k}$, while
 $P$ and $P^\perp$ are the corresponding orthogonal projections, $S_{\varepsilon,K}(\Pin)
 =\{x\in\R^{n+k}:|P^\perp x|\leq\varepsilon|Px|+K\}$ and $\varepsilon$ is the aperture.    For a stationary integral varifold $\Sigma^n\subset \R^{n+k}$, set $V(R)=\text{Vol} (B_R \cap \Sigma)$.  
 
\begin{theorem}[Varifold  thresholds; Figure \ref{f:real-growth-regimes}] 	\label{t:realsetting}
The three regimes for
 stationary integral varifolds:
\begin{enumerate}
\item[(i)] If the height of $\Sigma$ over $\Pi$ grows sublinearly, then $V (r)$ has Euclidean growth.
\item[(ii)] There exist $\delta_n , \delta_n'> 0$ so that if $\Sigma \subset S_{\varepsilon,K}(\Pin)$ and  $\varepsilon < \delta_n$, then 
$V(r)$ grows polynomially;  if $\varepsilon < \delta_n' $, then $V(r) \leq c\, (1+r)^{ n+c_n\,\varepsilon^2}$ for $c=c(\Sigma)>0$.  
\item[(iii)]  If $\varepsilon > \delta_n$, then there are examples in $S_{\varepsilon,K}(\Pin)$ with super-polynomial growth.
\end{enumerate}
\end{theorem}

\vskip1mm
\noindent
Theorem~\ref{t:catenoid-sum} shows that $\delta_n\le\max_{t\ge0}t/\cosh t\approx0.66$
in every codimension: sums of scaled catenoids, with large multiplicities, give
examples with arbitrarily fast growth.
Theorem \ref{t:varifold} shows that the dependence on $\varepsilon$ in (ii) is sharp. Given $\varepsilon \in (0,1/10)$:
\begin{enumerate}
\item[(iv)] There is a
stationary integral $2$-varifold in $S_{\varepsilon,1}(\Pi) \subset \R^3$ with
 $ V(r) \ge r^{2+\frac{4}{5} \, \varepsilon^2}$ for $r > R_0$.
\end{enumerate}

  In particular,  the confinement results of
\cite{CM1,CM2} cannot be improved from sublinear to linear height without further hypotheses.  
Figure~\ref{f:varifold3d} shows half of the example in (iv).

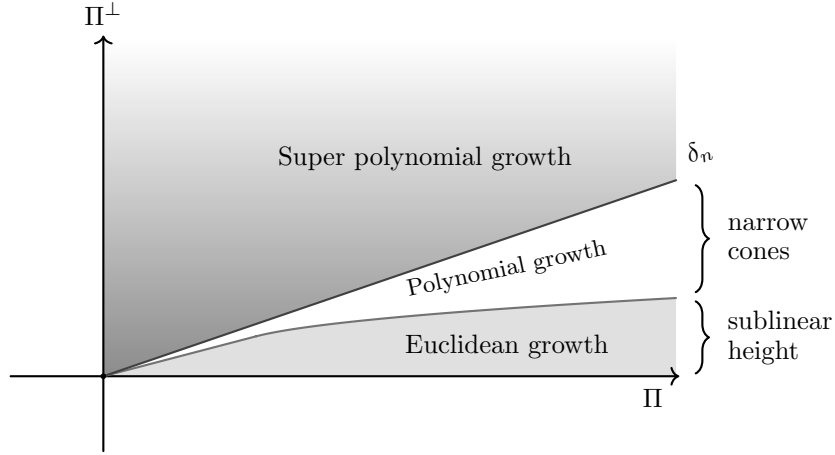
\begin{figure}[ht]
\centering
\begin{tikzpicture}[
    x=1.76cm,
    y=1.15cm,
    scale=.82,
    every node/.style={font=\small},
    line cap=round,
    line join=round
]
    \coordinate (O) at (0,0);

    \fill[black!12]
       (O)--(1.4,.56)--plot[domain=1.4:5.25,samples=100]
          (\x,{.56+.3394*(sqrt(\x-1.22)-sqrt(.18))})--(5.25,0)--(O);
    \shade[bottom color=black!45,top color=white]
       (O)--(5.25,2.75)--(5.25,4.70)--(0,4.70)--cycle;

    \draw[->,thick]
       (-.85,0)--(5.28,0);
    \node[below right] at (4.85,-.03) {$\Pi$};
    \draw[->,thick]
       (0,-1.05)--(0,4.78) node[above] {$\Pi^\perp$};

    \draw[thick,black!55]
       (O)--(1.4,.56)
       plot[domain=1.4:5.25,samples=100]
          (\x,{.56+.3394*(sqrt(\x-1.22)-sqrt(.18))});
    \draw[thick,black!75] (O)--(5.25,2.75);
    \node[rotate=20,above=2pt] at (5.55,2.8) {$\delta_n$};
    \fill (O) circle (1.3pt);

    \node[align=center] at (3.7,.42)
       {Euclidean growth};
    \node[align=center,rotate=12,font=\footnotesize] at (3.7,1.50)
       {Polynomial growth};
    \node[align=center] at (2.95,3.05)
       {Super polynomial growth};

    \draw[decorate,
          decoration={brace,amplitude=5pt,mirror},
          thick]
       (5.43,.04)--(5.43,1.05)
       node[midway,right=8pt,align=left]
       {sublinear\\[-1pt] height};

    \draw[decorate,
          decoration={brace,amplitude=5pt,mirror},
          thick]
       (5.43,1.18)--(5.43,2.68)
       node[midway,right=8pt,align=left]
       {narrow\\[-1pt] cones};
\end{tikzpicture}
\caption{Three regimes of volume-growth for \emph{stationary varifolds}.}
\label{f:real-growth-regimes}
\end{figure}

\begin{figure}[ht]
\centering
\includegraphics[width=0.5\textwidth]{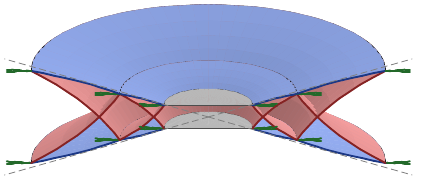}
\caption{Half of the stationary varifold in (iv), cut along a vertical plane
through the $x_3$-axis.  It consists of two flat disks (gray) and pieces of
catenoids (blue and red) that meet along circles and stay in the cone
$|x_3|\le\varepsilon r$ (dashed); the catenoidal ends are only indicated, in
green, on the cut.  See Section~\ref{s:varifolds}.}
\label{f:varifold3d}
\end{figure}

\vskip1mm
\subsection{Organization}
  Section \ref{s:tangent}  proves a lower bound
for the aperture of a  cone containing a minimal surface in terms of 
   its tilt-excess.  Section \ref{s:kahler} proves
    optimal lower and upper bounds for the tilt-excess of a holomorphic curve in terms of
     the K\"ahler-angle.  Section \ref{s:examples} first constructs superquadratic Lagrangian disks  and then obtains the angle-dependent existence theorem by
transporting them with a suitable complex-linear isomorphism.  Section \ref{s:rapid} modifies this construction to get arbitrarily fast area growth.
Section
\ref{s:algebraicity} proves the angle-dependent algebraicity theorem and deduces part (ii) in Theorem
\ref{t:angleversion}. 
  Section \ref{s:narrow} gives the
polynomial bound for narrow real cones,  the catenoid examples to establish Theorem \ref{t:realsetting}, and a holomorphic analog of catenoids.
Section \ref{s:varifolds}
proves (iv) stated after Theorem \ref{t:realsetting}.
Section
\ref{s:higher-domains} extends the holomorphic results to arbitrary complex dimension and
codimension.

\section{A  lower bound for tilt excess}\label{s:tangent}

The purpose of this section is to turn a uniform lower bound on the tilt-excess of a minimal surface into a lower bound
for the aperture of any cone containing it.  Proposition
\ref{p:tangent-obstruction} gives the precise bound in $\R^n$, using only
minimality and no complex or holomorphic structure.

Let $\Pin\subset\R^n$ be a two-plane, and write $P,\,P^\perp$ for the orthogonal projections
onto $\Pin$ and $\Pout$, respectively.

\begin{definition}[Average vertical
fraction  of $V$]\label{d:mu}
For a two-plane $V\subset\R^n$ and an orthonormal basis $v_1,v_2$ of $V$, set
\[
 \mu_\Pin(V)= \frac{1}{2} \, \Tr_V \, \left( v\mapsto|P^\perp v|^2 \right)
 =\tfrac12\bigl(|P^\perp v_1|^2+|P^\perp v_2|^2\bigr) \, .
\]
\end{definition}

Since $\mu_{\Pin}$ is defined as a trace, it is independent of the choice of orthonormal basis.
In applications below, the plane $\Pi$ will be fixed, while the plane $V$ will be the tangent plane at a point to a holomorphic curve, and thus $V$ will vary with the particular point.
 
 \vskip1mm
 When $\Sigma$ is a submanifold, the average vertical fraction $\mu_{\Pi}$ of its tangent plane $T\Sigma$ at a point is   the 
 tilt-excess\footnote{Our normalization differs by a factor from
\cite{CM1}.  In the literature, the term \emph{tilt excess} refers to
the integral over a ball of the pointwise quantity defined here.  We use only
the pointwise quantity and therefore adopt the same term for it. If $\Sigma^m$ is $m$-dimensional, then $E_{\Pi}$ here will be $\frac{1}{m}$ of the quantity in \cite{CM1}.}; cf. \cite{Al, Be, CM1, dL}.  The tilt-excess $E_{\Pi}$ at a point $x$ 
 measures
  the deviation of $T_x\Sigma$ from $\Pi$.  If $e_3 , \dots , e_n$ is an orthonormal basis for $\Pout$, then 
 \begin{align}
 	E_{\Pin} = \frac{1}{2} \,  \sum_{j=3}^n \left| e_j^T \right|^2  
 \end{align}
 where $e_j^T$ is the projection of $e_j$ onto the tangent plane $T_x\Sigma$.
 
 \begin{lemma}	\label{l:Emu}
 If $\Sigma$ is a submanifold, then $E_{\Pin} = \mu_{\Pin}(T\Sigma)$.
 \end{lemma}
 
 \begin{proof}
 Let $e_3, \dots , e_n$ be an orthonormal frame for $\Pout$ and let $v_1 , v_2$ be an orthonormal basis for the tangent space to $\Sigma$ at a point.
 We have that
 \begin{align}
 	2\, E_{\Pin} =   \sum_{j=3}^n \left| e_j^T \right|^2  =   \sum_{j=3}^n \sum_{m=1}^2 \langle e_j , v_m \rangle^2  
	=   \sum_{m=1}^2  \sum_{j=3}^n \langle e_j , v_m \rangle^2  =   \sum_{m=1}^2  \left| P^{\perp} v_m \right|^2  = 2\, \mu_{\Pin}(T\Sigma) \, .
 \end{align}
 \end{proof}

\vskip1mm

For $\varepsilon,K\geq0$ use the same notation as before,
\[
 S_{\varepsilon,K}(\Pin)
 =\{x\in\R^n:|P^\perp x|\leq\varepsilon|Px|+K\}.
\]
For $\kappa\geq0$ define, in analogy with \eqref{e:rho},
\[
 \rho^\Pin_\kappa(x)=|Px|^2-\kappa|P^\perp x|^2 .
\]

The next lemma should be compared with Section $1$ in \cite{CM1}.

\begin{lemma} \label{l:minimal-bochner}
Let $X:\Sigma\to\R^n$ be a conformal harmonic map from a Riemann surface.  In isothermal coordinates
 $\zeta=(\xi, \eta)$, put $\lambda=|X_\xi|=|X_\eta|$.  At every regular
point, we have
\begin{equation}\label{e:minimal-bochner}
\Delta_{\Sigma} \bigl(\rho^\Pin_\kappa\circ X\bigr)
 =4\lambda^2\Bigl[ 1- (1+\kappa) \, E_\Pin   \Bigr].
\end{equation}
At a branch point, $\Delta_{\Sigma}  (\rho^\Pin_\kappa\circ X)=0$.
\end{lemma}

\begin{proof}
We will assume that $X$ is a minimal isometric immersion, i.e., $\lambda = 1$; the general conformal and harmonic case follows 
with easy modifications.   Let $e_1 , e_2$ be an orthonormal basis for $\Pin$ and $e_3 , \dots , e_n$ be an orthonormal basis for $\Pout$.
As above, let $v_1,v_2$ be an orthonormal basis for $\Sigma$ at a point and let $e^T = \langle e , v_1 \rangle \, v_1 + \langle e , v_2 \rangle \, v_2$
be tangential projection of a vector $e$.

Since the coordinate functions are harmonic and $\Delta_{\Sigma} |x|^2 = 4$ for a minimal surface, we have
\begin{align}
	\Delta_{\Sigma} \,  \rho^\Pin_\kappa &= 4 - (1+\kappa) \, \Delta_{\Sigma} \, \sum_{j=3}^n \langle x , e_j \rangle^2 =
	 4 - 2\,  (1+\kappa) \,   \sum_{j=3}^n |e_j^T|^2  =   4 - 4\,  (1+\kappa) \, E_{\Pi}   \, ,  
\end{align}
giving the claim.
\end{proof}

The lemma applies to every oriented minimal immersion after the domain is given its induced
conformal structure, and it also allows branch points.

\begin{proposition} \label{p:tangent-obstruction}
Let $X:\Sigma\to\R^n$ be a nonconstant proper conformal harmonic map from an open
Riemann surface, and suppose that at every regular point
$
 E_\Pin \geq\mu_0
$
for some $0<\mu_0 < 1$.  If
$X(\Sigma)\subset S_{\varepsilon,K}(\Pin)$ for some $\varepsilon,K\geq0$,
 then
\[
 \varepsilon\geq\sqrt{\frac{\mu_0}{1-\mu_0}} \, .
\]
\end{proposition}

\begin{proof}
Suppose that
$\varepsilon<\sqrt{\mu_0/(1-\mu_0)}$. Choose $\varepsilon_0$ with
\[
 \varepsilon<\varepsilon_0<\sqrt{\frac{\mu_0}{1-\mu_0}}
 \qquad\text{and set}\qquad \kappa=\varepsilon_0^{-2},
\]
so that  $\kappa>(1-\mu_0)/\mu_0$.  Thus, at every
regular point,  we have
\[
 1-(1+\kappa) E_{\Pin}
 \leq1-(1+\kappa)\mu_0<0.
\]
Lemma \ref{l:minimal-bochner} therefore shows that
$\rho:=\rho^\Pin_\kappa\circ X$ is superharmonic on all of $\Sigma$.  The cone bound
$|P^\perp X|\leq\varepsilon|PX|+K$ gives
\begin{align}	\label{e:CB1}
 \rho \geq\Bigl(1-\bigl(\tfrac{\varepsilon}{\varepsilon_0}\bigr)^2\Bigr)|PX|^2
 -\frac{2\varepsilon K}{\varepsilon_0^2}|PX|-\frac{K^2}{\varepsilon_0^2},
\end{align}
while
\begin{align}	\label{e:CB2}
 |X|^2\leq(1+2\varepsilon^2)|PX|^2+2K^2.
\end{align}
Using the absorbing inequality $2ab \leq \delta \, a^2 + \delta^{-1} \, b^2$ with 
$$\delta = \frac{1}{2} \, \Bigl(1-\bigl(\tfrac{\varepsilon}{\varepsilon_0}\bigr)^2\Bigr) , \, a = |PX| , {\text{ and }}
b = \frac{\varepsilon \, K }{\varepsilon_0^2}
\, ,$$
 (\ref{e:CB1}) gives that
\begin{align}	\label{e:CB3}
 \rho \geq \frac{1}{2} \, \Bigl(1-\bigl(\tfrac{\varepsilon}{\varepsilon_0}\bigr)^2\Bigr)|PX|^2
 -\frac{K^2}{\varepsilon_0^2} - 2\, \Bigl(1-\bigl(\tfrac{\varepsilon}{\varepsilon_0}\bigr)^2\Bigr)^{-1} \, 
  \frac{\varepsilon^2 \, K^2 }{\varepsilon_0^4}
 \, .
\end{align}
Along every divergent sequence in $\Sigma$, properness forces $|X|\to\infty$; hence 
$|PX|\to\infty$ by (\ref{e:CB2}), and thus $\rho \to\infty$ by (\ref{e:CB3}).  

Exhaust $\Sigma$ by relatively compact domains
$\Omega_j$ with smooth boundary, chosen so that $\partial\Omega_j$ leaves every compact subset
of $\Sigma$ as $j\to \infty$.  In particular, the minimum of $\rho$ on $\partial \Omega_j$ goes to infinity as $j \to \infty$. Thus, 
for any point $p\in\Omega_1$, the minimum principle for superharmonic functions gives
\[
 \rho (p)\geq\min_{\partial\Omega_j}\rho \longrightarrow\infty,
\]
a contradiction.
\end{proof}

\section{Tilt-excess for  holomorphic curves}  \label{s:kahler}

 The purpose of this section
is to determine the optimal lower and upper bounds for the tilt-excess of a holomorphic curve relative to a fixed
two-plane $\Pin$.  These bounds depend only on the K\"ahler angle of $\Pin$ and
are attained by explicit complex lines.  The bounds are established in Lemma
\ref{l:mu}, and Corollary \ref{c:lines} translates them into sharp statements
about complex lines contained in, or missing, a cone.

Fix a constant orthogonal complex structure
$J$ on $\R^4$, thereby identifying $(\R^4,J)$ with $\C^2$. 
Let $\omega(x,y)=\langle Jx,y\rangle$ be the associated K\"ahler form.  A
\emph{complex line} is a $J$-invariant two-plane through $0$.  The  K\"ahler angle 
$\theta=\theta(\Pin)\in[0,\pi/2]$   of $\Pin$ is defined by $\cos \theta = |\omega (e_1 , e_2)|$ for an orthonormal basis $\{ e_1 , e_2 \}$ of
$\Pin$.  It is easy to check that $\theta$ does not depend on the choice of basis and that $\theta = 0$  for a complex line.

\subsection{The vertical fraction of a complex line}
This subsection will develop the linear algebra relating the average vertical fraction for a complex line 
to the K\"ahler angle of the plane $\Pin$.

\vskip1mm
We will need the following elementary normal form for a basis.

\begin{lemma}\label{l:normal}
Let $\Pin\subset\R^4$ be a two-plane with K\"ahler angle $\theta$.  Then there is an orthonormal
basis $f_1,Jf_1,f_2,Jf_2$ of $\R^4$ such that
\[
\Pin=\mathrm{span}\{e_1,e_2\},\quad e_1=f_1,\ \ e_2= \cos \theta \,Jf_1+\sin \theta \,f_2,
\]
\[
\Pout=\mathrm{span}\{e_3,e_4\},\quad e_3=\sin \theta \,Jf_1- \cos \theta \,f_2,\ \ e_4=Jf_2 .
\]
Moreover $\theta(\Pout)=\theta(\Pin)$.
\end{lemma}

\begin{proof}
We will assume that $0 < \cos \theta < 1$; the edge cases follow easily.
Choose an orthonormal basis $e_1,e_2$ of $\Pin$ with 
\begin{align}	\label{e:thetaf}
	\omega(e_1,e_2)=\langle Je_1,e_2\rangle=\cos \theta \, ;
\end{align}
replacing $e_2$ by $-e_2$ if necessary, this is possible.  Set $f_1=e_1$.  
Since
$e_1$ and $e_2$ are orthonormal, (\ref{e:thetaf}) uniquely determines a unit vector $f_2$ perpendicular to both $e_1 = f_1$ and $Jf_1$ 
with
\begin{align}
	e_2= \cos \theta \,Jf_1+ \sin \theta \, f_2 \,  .
\end{align}
The lemma follows by using these properties of $f_1$ and $f_2$ to establish that
\begin{itemize}
\item The unit vectors $f_1 , f_2 , Jf_1 , J f_2$ are orthogonal.
\item The unit vectors $\sin \theta \,Jf_1- \cos \theta \,f_2$ and $J \, f_2$ are orthogonal.
\item The vectors $\sin \theta \,Jf_1- \cos \theta \,f_2$ and $J \, f_2$ are orthogonal to   $f_1$ and 
$ \cos \theta \,Jf_1+\sin \theta \,f_2$. 
\end{itemize}	 
\end{proof}

If $L$ is a complex line and $v\in L$ is a unit vector, Definition \ref{d:mu} becomes
\[
 \mu_\Pin(L)=\tfrac12\bigl(|P^\perp v|^2+|P^\perp Jv|^2\bigr).
\]
Given $\Pi$ with K\"ahler angle $\theta$, let $f_1 , f_2$ be from Lemma \ref{l:normal} and define unit  
\begin{align}	\label{e:vmin}
	v_-  =\cos\frac\theta2\,Jf_1+\sin\frac\theta2\,f_2  {\text{ and }}
	v_+  =\cos\frac\theta2\,f_2-\sin\frac\theta2\,J \, f_1 \, .
\end{align}
Let $L_{\pm}$ be the complex lines generated by $v_{\pm}$, so $\{ v_- , J\, v_-\}$ is a basis for $L_-$ and similarly for $L_+$.

\begin{lemma}\label{l:mu}
 If $\Pi$, $v_{\pm}$ and $L_{\pm}$ are as above, then 
\begin{enumerate}
\item[(i)] $L_-\perp L_+$, $\R^4=L_-\oplus L_+$, and  
\[
|P^\perp x|^2= \sin^2\frac\theta2\,|x|^2 {\text{ for $x \in L_-$ while }}
|P^\perp x|^2= \cos^2\frac\theta2\,|x|^2 {\text{ for }} x \in L_+.
\]
\item[(ii)] For \emph{every} complex line $L$,
\[
\sin^2\tfrac\theta2\ \le\ \mu_\Pin(L)\ \le\ \cos^2\tfrac\theta2,
\]
with equality on the left exactly for $L=L_-$ and on the right exactly for $L=L_+$ (when
$\theta<\pi/2$; for $\theta=\pi/2$ all complex lines have $\mu_\Pin\equiv\frac12$).
\end{enumerate}
\end{lemma}

\begin{proof}
The vectors $v_- , v_+, J \, v_-$ and $Jv_+$ are easily seen to be orthonormal, giving the first part of (i).
 Using Lemma \ref{l:normal},
it is easy to check that
\begin{align}
	v_- &= \sin \frac{\theta}{2} \, e_3 + \cos \frac{\theta}{2} \, e_2 {\text{ and }}
	J\, v_-  = \sin \frac{\theta}{2} \, e_4 - \cos \frac{\theta}{2} \, e_1 \, , \\
	v_+ &= - \cos \frac{\theta}{2} \, e_3 + \sin \frac{\theta}{2} \, e_2  {\text{ and }}
	J\, v_+  =  \cos \frac{\theta}{2} \, e_4 + \sin \frac{\theta}{2} \, e_1 \, .
\end{align}
The second part of claim (i) follows immediately from this.  Moreover, 
 we see for a vector 
$$
	v = a_1 \, v_- + a_2 \, J \, v_- + a_3 \, v_+ + a_4 \, J \, v_+ 
$$
that 
\begin{align}
	 P^{\perp} v &= \Bigl( a_1 \, \sin \frac{\theta}{2} - a_3 \, \cos \frac{\theta}{2} \Bigr) \, e_3 +
	  \Bigl( a_2 \, \sin \frac{\theta}{2} + a_4 \, \cos \frac{\theta}{2} \Bigr) \, e_4 \, , \\
	  P^{\perp} \, J \, v &= \Bigl( -a_2 \, \sin \frac{\theta}{2} + a_4 \, \cos \frac{\theta}{2} \Bigr) \, e_3 +
	  \Bigl( a_1 \, \sin \frac{\theta}{2} + a_3 \, \cos \frac{\theta}{2} \Bigr) \, e_4 \, .
\end{align}
From this, we see that
 \begin{align}
 	\frac{1}{2} \, \Bigl( \left|  P^{\perp} v \right|^2 +   \left|  P^{\perp} \, J \, v \right|^2 \Bigr) =
	(a_1^2 + a_2^2) \, \sin^2 \frac{\theta}{2}+ (a_3^2 +a_4^2) \, \cos^2 \frac{\theta}{2} \, .
 \end{align}
 Claim (ii) follows immediately from this.
\end{proof}

\begin{corollary}\label{c:lines}
If $\Pin$ has K\"ahler angle $\theta$ and $\varepsilon\ge0$, then:
\begin{enumerate}
\item[(i)] If  
$\varepsilon <\tan\frac\theta2$, then $C_\varepsilon(\Pin)$  does not contain any complex lines.
\item[(ii)]  If $\varepsilon\ge\tan\frac\theta2$, then  $C_\varepsilon(\Pin)$ contains $L_-$.
\item[(iii)] If $\varepsilon<\cot\frac\theta2$, then $L_+ \cap C_\varepsilon(\Pin)=\{0\}$.
\item[(iv)]   If  
$\varepsilon \ge \cot\frac\theta2$, then every complex line $L$ intersects 
$ C_\varepsilon(\Pin) \setminus \{0\}$.
\end{enumerate}
\end{corollary}

\begin{proof}
For a unit vector $x$, $|Px|^2+|P^\perp x|^2=1$, so
\[
x\in C_\varepsilon(\Pin)\iff |P^\perp x|^2\le\frac{\varepsilon^2}{1+\varepsilon^2}.
\]
Claims (ii) and (iii) follow immediately from part (i) in Lemma \ref{l:mu}.
If  $L\subset C_\varepsilon(\Pin)$, then 
$\mu_\Pin(L)\le\varepsilon^2/(1+\varepsilon^2)$; by Lemma \ref{l:mu}(ii),
$\sin^2\frac\theta2\le\varepsilon^2/(1+\varepsilon^2)$, i.e.\ $\varepsilon\ge\tan\frac\theta2$. This gives (i).
  If $L\cap C_\varepsilon(\Pin)=\{0\}$ then $|P^\perp x|^2>\varepsilon^2/(1+\varepsilon^2)$ for
all unit $x\in L$.  Averaging this strict pointwise inequality gives
$\mu_\Pin(L)>\varepsilon^2/(1+\varepsilon^2)$.  Hence
$\varepsilon^2/(1+\varepsilon^2)<\cos^2\frac\theta2$ by Lemma \ref{l:mu}(ii), i.e.\
$\varepsilon<\cot\frac\theta2$.  This gives (iv).
\end{proof}

\subsection{Sharp holomorphic  thresholds} 
We will now apply the linear algebra of the previous subsection to the tangent plane for a holomorphic curve and combine this with 
the results from Section \ref{s:tangent}.

The next corollary shows that  the quadratic
cone function $\rho^\Pin_\kappa$ is subharmonic for any holomorphic curve if
$\kappa\leq\tan^2(\theta/2)$ and superharmonic if
$\kappa\geq\cot^2(\theta/2)$, with both ranges sharp.

\begin{corollary}\label{l:bochner}
Let   $X:\Sigma\to\C^2$ a nonconstant holomorphic map.  Then
\[
\rho^\Pin_\kappa\circ X\ \text{ is subharmonic if }\ \kappa\le\tan^2\tfrac\theta2,
\qquad
\rho^\Pin_\kappa\circ X\ \text{ is superharmonic if }\ \kappa\ge\cot^2\tfrac\theta2 \, .
\]
\end{corollary}

\begin{proof}
Since conformal changes preserve the sign of the Laplacian in dimension two, we can assume that the map is an isometry (the case of branch points follows with easy modifications).  Lemma \ref{l:minimal-bochner}
gives that
\begin{equation} 
\Delta_{\Sigma} \bigl(\rho^\Pin_\kappa\circ X\bigr)
 =4 \, \Bigl[ 1- (1+\kappa) \, E_\Pin   \Bigr].
\end{equation}
If $ \kappa\le\tan^2\tfrac\theta2$, then  Lemma \ref{l:mu}(ii) gives that
\begin{align}
	1- (1+\kappa) \, E_\Pin \geq 1- \bigl(1+\tan^2 \frac{\theta}{2} \bigr) \, \cos^2 \frac{\theta}{2} = 0 \, .
\end{align}
If $ \kappa\ge\cot^2\tfrac\theta2$, then  Lemma \ref{l:mu}(ii) gives that
\begin{align}
	1- (1+\kappa) \, E_\Pin \leq 1- \bigl(1+\cot^2 \frac{\theta}{2} \bigr) \, \sin^2 \frac{\theta}{2} = 0 \, .
\end{align}
\end{proof}

\vskip1mm
We now prove nonexistence of holomorphic curves for $\varepsilon$ below the threshold $\tan \frac{\theta}{2}$
  (this addresses (i) in the introduction).

\begin{corollary}\label{p:nonexist}
Let $\Sigma$ be an open Riemann surface, $F:\Sigma\to\C^2$ a nonconstant proper
holomorphic map, and $\Pin$ a two-plane of K\"ahler angle $\theta$.  If
$F(\Sigma)\subset S_{\varepsilon,K}(\Pin)$ for some $K<\infty$, then
$\varepsilon\ge\tan\frac\theta2$. \end{corollary}

\vskip1mm
 In particular, for $\theta=\pi/2$ there is no proper holomorphic
curve contained in $S_{\varepsilon,K}(\Pin)$ with $\varepsilon<1$ and $K<\infty$ over a
Lagrangian plane.
\vskip1mm

\begin{proof}
The corollary follows by
arguing as in the first part of Corollary \ref{l:bochner} and applying
 Proposition \ref{p:tangent-obstruction}.
\end{proof}

The threshold in  the corollary is sharp since there is a complex line on the boundary of the critical cone.

\section{Superquadratic holomorphic disks}\label{s:examples}

In this section, we will prove a weak version of (iii) in Theorem \ref{t:angleversion}: Once we get above the threshold $\cot \frac{\theta}{2}$, there are 
holomorphic immersions with faster than quadratic area growth.  We will later show that the area growth can be arbitrarily fast.

We begin with the Lagrangian case.

\begin{theorem}\label{t:main}
Let $\Pin\subset(\R^4,J)$ be Lagrangian.  For every $\varepsilon>1$ there is a proper
$J$-holomorphic immersion $X:\D\to\R^4$ such that $X(\D)\subset C_\varepsilon(\Pin)$ and
\begin{equation}\label{e:growth}
 \lim_{R\to\infty}\frac{A_X(R)}{R^2\log R}=\infty .
\end{equation}
The induced metric is complete and the surface is 
area minimizing.
\end{theorem}

The proper holomorphic disk in the cone is provided by \cite{FG}.  The additional
points here are that, by first working in a slightly narrower cone, the disk can
be perturbed to an immersion without leaving the prescribed cone, and that the
resulting proper immersion satisfies the area lower bound \eqref{e:growth}.

\subsection{An area lower bound}\label{s:lower}

 A manifold is parabolic if every positive superharmonic
function is constant; equivalently, it admits no positive Green's function
\cite{A}.  Otherwise, it is nonparabolic.

\begin{lemma}[Nonparabolic area growth]\label{l:nonparabolic-growth}
Let $X:\Sigma^2\to\R^n$ be a proper minimal immersion without boundary.  If  $\Sigma$ is
nonparabolic, then, for every $p\in\R^n$,
\[
 \lim_{R\to\infty}
 \frac{\Area\bigl(X^{-1}(B_R(p))\bigr)}{R^2\log R}=\infty .
\]
\end{lemma}

\begin{proof}
Properness and the absence of boundary imply that the induced metric on $\Sigma$ is complete.
It suffices to assume that 
  $p\in \Sigma$.  Let $\cA_p(r)$ be the area of the intrinsic ball $B_r^{\Sigma}(p)$.  Set
\[
 A_p(R)=\Area\{q\in \Sigma:|X(q)-X(p)|<R\}.
\]
Euclidean distance is no greater than intrinsic distance, so
$B_r^{\Sigma}(p)\subset X^{-1}(B_r(X(p)))$ and hence $\cA_p(r)\leq A_p(r)$.  A result of Varopoulos \cite{Varopoulos} says that a complete nonparabolic surface has 
$\int_1^\infty r\,dr/\cA_p(r)<\infty$.  The preceding
comparison gives
\begin{equation}\label{e:vol}
 \int_1^\infty\frac{r\,dr}{A_p(r)}<\infty.
\end{equation}

By the monotonicity formula (see, for example, \cite[Corollary~1.13]{CMbook}),
$\Theta_p(R)=A_p(R)/R^2$ is nondecreasing.  Substituting $R=e^t$ in \eqref{e:vol} shows that the
nondecreasing function $\Phi(t)=\Theta_p(e^t)$ satisfies
\begin{equation}\label{e:vol3}
 \int_0^\infty\frac{dt}{\Phi(t)}<\infty.
\end{equation}
We claim that $\Phi(t)/t\to\infty$.  Otherwise there are $C<\infty$ and $t_j\to\infty$ such that,
after passing to a subsequence, $t_{j+1}\geq2t_j$ and $\Phi(t_j)\leq Ct_j$.  Monotonicity gives
\[
 \int_{t_j/2}^{t_j}\frac{dt}{\Phi(t)}
 \geq\frac{t_j/2}{\Phi(t_j)}\geq\frac1{2C}.
\]
The intervals $[t_j/2,t_j]$ are pairwise disjoint, contradicting \eqref{e:vol3}.  The
claim follows.
\end{proof}

\subsection{The Forstneri\v c--Globevnik input}   \label{s:cone}

Identify $\C^2$ with $\R^4$ by writing
\begin{equation}\label{e:ident}
(z_1,z_2)=(x_1+ix_3,\ x_2+ix_4)
\end{equation}
and define the complex structure $J$ by
\[
 J\partial_1=\partial_3,\qquad J\partial_2=\partial_4,\qquad
 J\partial_3=-\partial_1,\qquad J\partial_4=-\partial_2.
\]
The linear map $J$ is orthogonal and $J^2$ is minus the identity.
Note that under \eqref{e:ident} the plane $\Pin=\{x_3=x_4=0\}$ and its
orthogonal complement $\Pout=\{x_1=x_2=0\}$ are Lagrangian, i.e.\ they have K\"ahler angle $\pi/2$, and
$J\Pin=\Pout$.  For $\varepsilon>0$ define
\begin{equation}\label{e:rho}
\rho_\varepsilon(z_1,z_2)
=x_1^2+x_2^2-\varepsilon^{-2}\bigl(x_3^2+x_4^2\bigr),
\end{equation}
and set
\[
  \mathring C_\varepsilon(\Pin) =\{z\in\C^2:\rho_\varepsilon(z)>0\}.
\]
Thus $\mathring C_\varepsilon(\Pin)$ is exactly the interior of $C_\varepsilon(\Pin)$.

We use the following form of \cite[Theorem~1.2]{FG}, adapted to our application.

\begin{theorem}[Forstneri\v c--Globevnik]\label{t:FG}
Let $\varepsilon>1$, let $p\in \mathring C_\varepsilon(\Pin)$, and fix $0<r<1$ and $\tau>0$.  There is a
proper holomorphic map $F:\D\to\C^2$ such that
\[
\sup_{|\zeta|\le r}|F(\zeta)-p|<\tau
\]
and
\begin{equation}\label{e:FG}
\rho_\varepsilon\circ F>\rho_\varepsilon(p)-\tau \ \text{ on }\D,
\qquad
\rho_\varepsilon(F(\zeta))\longrightarrow\infty\ \text{ as }|\zeta|\to1 .
\end{equation}
Conversely, if $0<\varepsilon\leq1$, then $\rho_\varepsilon\circ F$ is unbounded below for every
proper holomorphic map $F:\D\to\C^2$.
\end{theorem}

To see why Theorem \ref{t:FG} follows from \cite{FG}, set $c=\varepsilon^{-2}<1$, choose
$M<\rho_\varepsilon(p)$, and take the constant disk $h\equiv p$.  The cited result
\cite[Theorem~1.2]{FG} asserts that conclusions (i)--(iii) of
\cite[Theorem~1.1]{FG} hold in this situation; these give,
respectively, the boundary divergence, the global lower bound, and the approximation on the
smaller disk.
In particular, the first two conclusions do not follow merely from approximation on compact
subsets.  Taking $p=(1,0)$ and, for example, $\tau=\frac12$, the resulting map satisfies
\begin{equation}\label{e:FGpositive}
\rho_\varepsilon\circ F\ge\frac{1}{2} \quad\text{on }\D\, .
\end{equation}

\subsection{Proof of Theorem \ref{t:main}}

\begin{proof}[Proof of Theorem \ref{t:main}]
Since $\varepsilon > 1$, we can choose $\varepsilon_0$ with $1< \varepsilon_0 < \varepsilon$.   Set 
$\kappa = \varepsilon^{-2}$  and $\kappa_0 = \varepsilon_0^{-2}$. 
Apply Theorem \ref{t:FG} with aperture $\varepsilon_0$, $p=(1,0)$ in $\Pi$, and $\tau = \frac{1}{2}$ to get
a proper holomorphic map $G:\D \to \C^2$ with 
\begin{align}	\label{e:FG1}
	\rho_{\kappa_0}^{\Pin} \circ G &\geq \rho_{\kappa_0}^{\Pin} (1,0) - \frac{1}{2} = \frac{1}{2} \, , \\
	\rho_{\kappa_0}^{\Pin} \circ G (\zeta) &\to \infty {\text{ as }} |\zeta| \to 1 \, . \label{e:FG2}
\end{align}
The latter implies that $\left| P\, G(\zeta) \right| \to \infty$ as $|\zeta| \to 1$. 

The cone $C_{\varepsilon_0}(\Pin)$ is  the set  $\{ \rho_{\kappa_0}^{\Pin} \geq 0\}$. Thus, the proper holomorphic
 map $G$ is contained in the cone $C_{\varepsilon_0}(\Pin) \subset C_{\varepsilon}(\Pin)$.  However, $G$ might not be an immersion since there may be points $\zeta \in \D$ where $G'(\zeta) = 0$.  We will make a small linear perturbation of $G$ to eliminate any zeros, giving an immersion, 
 and show that this perturbation is in the larger cone $ C_{\varepsilon}(\Pin)$.
 
 The image of the map $G':\D \to \C^2$ is a countable union of smooth images of compact two-dimensional sets, so it has four dimensional Lebesgue measure zero.  Thus, it cannot contain any neighborhood of $0 \in \C^2$, and we can choose a sequence $\xi_i \in \C^2$ so that
 $$
 	|\xi_i| \to 0 {\text{ and }} -\xi_i \notin G'(\D) \, .
 $$
 It follows that the maps $G_i = G + \zeta \, \xi_i$ are holomorphic immersions.  Since $|\zeta \, \xi_i| \to 0$, these maps are also proper. 
 
  It remains to show that $\rho_{\kappa}^{\Pin} \circ G_i > 0$ for $i$ sufficiently large. Using the absorbing inequality, we have
  for any $\delta > 0$ that 
  \begin{align}
  	\rho_{\kappa}^{\Pin} \circ G_i   &= \left| P \, G_i  \right|^2 - \kappa \,  \left| P^{\perp} \, G_i  \right|^2 \notag \\
	&\geq  (1-\delta) \, \left| P \, G \right|^2 - \frac{|\xi_i|^2}{\delta} 
	- \kappa \,  (1+\delta) \, \left| P^{\perp} \, G  \right|^2 - \kappa \, \left( |\xi_i|^2 + \frac{|\xi_i|^2}{\delta} \right) \\
	&\geq (1-\delta) \, \left| P \, G \right|^2 - \frac{|\xi_i|^2}{\delta} 
	- \frac{\kappa}{\kappa_0}\,  (1+\delta) \, \left| P \, G  \right|^2 - \kappa \, \left( |\xi_i|^2 + \frac{|\xi_i|^2}{\delta} \right) \, , \notag
  \end{align}
  where the last inequality used that $	\rho_{\kappa_0}^{\Pin} \circ G \geq 0$.  Since $\kappa < \kappa_0$, we can choose $\delta > 0$ sufficiently small to get a constant $\beta > 0$ so that 
   \begin{align}
  	\rho_{\kappa}^{\Pin} \circ G_i   \geq \beta\, \left| P \, G \right|^2 - \frac{|\xi_i|^2}{\beta}  \, .
  \end{align}
  Since we have the lower bound $ \left| P \, G \right|^2 \geq \frac{1}{2}$ and $\xi_i \to 0$, this completes the existence proof.
  
  The area growth (\ref{e:growth})
  follows from Lemma  \ref{l:nonparabolic-growth}
since the surface is conformal to a disk and thus nonparabolic.  It is area-minimizing by Wirtinger's inequality.  Finally, it is complete since it is a proper immersion.
\end{proof}

\subsection{Super-quadratic growth for arbitrary K\"ahler angle}

Using the next lemma, the previous argument generalizes to give examples  in $S_{\varepsilon,K}$ for any $\varepsilon  > \cot \frac{\theta}{2}$ for arbitrary K\"ahler angle $\theta$.

\begin{lemma}[Transport from the Lagrangian case]\label{l:transport}
Let $\Pin_0=\R^2\subset\C^2$ be the standard Lagrangian plane, and let
$\Pin\subset\C^2$ be a two-plane of K\"ahler angle
$0<\theta\leq\pi/2$.  If $\eta>0$ and $\eta\cos\theta<1$, then there is
a complex-linear isomorphism $T\in GL(2,\C)$ such that
\[
 T(\Pin_0)=\Pin
 \qquad\text{and}\qquad
 T\bigl(C_\eta(\Pin_0)\bigr)
 \subset C_{\nu_\theta(\eta)}(\Pin),
\]
where
\[
 \nu_\theta(\eta)
 =\frac{\eta\sin\theta}{1-\eta\cos\theta}.
\]
Moreover, if $f:\Sigma\to\C^2$ is a complete proper holomorphic
immersion, then $T\circ f$ has the same properties, and
\[
 \frac{A_f(R)}{R^2\log R}\longrightarrow\infty
 \quad\Longrightarrow\quad
 \frac{A_{T\circ f}(R)}{R^2\log R}\longrightarrow\infty.
\]
\end{lemma}

\begin{proof}
Choose a real isometry $B:\Pin_0\to\Pin$ and extend it complex-linearly by
\[
 T(x+iy)=Bx+JBy,\qquad x,y\in\Pin_0.
\]
Since $\theta>0$, we have that 
$\Pin\cap J\Pin=\{0\}$.  It follows that $T$ is injective and hence is in 
$GL(2,\C)$.

Let $P$ and $P^\perp$ be the orthogonal projections onto $\Pin$ and
$\Pin^\perp$. Suppose that $x$, $y\in \Pi_0$ and $x+iy\in C_\eta(\Pin_0)$ so $|y|\leq\eta|x|$.  We need to show that $T(x+iy)\in C_{\nu_\theta(\eta)}(\Pin)$.  Set $p=Bx$ and $q=By$, then since $B$ is an isometry 
\[
 |PJq|=\cos\theta\,|q|=\cos\theta\,|y|,
 \qquad
 |P^\perp Jq|=\sin\theta\,|q|=\sin\theta\,|y|\leq\eta\sin\theta\,|x|.
\]
 and
\[
 |P^\perp T(x+iy)|=|P^\perp(p+J\,q)|
 =|P^\perp Jq|
 \leq\eta\sin\theta\,|x|,
\]
whereas
\[
 |PT(x+iy)|
 =|p+PJq| \geq |p|-\cos\theta\,|y|
 = |x|-\cos\theta\,|y|
 \geq(1-\eta\cos\theta)|x|.
\]
Since $\eta\cos\theta<1$, these inequalities give
\[
 |P^\perp T(x+iy)|
 \leq \nu_\theta(\eta)\,|PT(x+iy)|,
\]
which proves the cone inclusion.

Now let $F=T\circ f$.  Since $T$ is an invertible complex-linear map,
$F$ is again a proper holomorphic immersion.  Set 
$m=\|T^{-1}\|^{-1}$ and $M=\|T\|$,  
then
\[
 m|z|\leq|Tz|\leq M|z|.
\]
The corresponding induced metrics are uniformly equivalent, so completeness
is preserved.  Moreover,
\[
 f^{-1}\bigl(B_{R/M}(0)\bigr)
 \subset F^{-1}\bigl(B_R(0)\bigr),
\]
and the two-dimensional Jacobian of $T$ on every plane is at least $m^2$.
Therefore
\[
 A_F(R)\geq m^2 A_f(R/M).
\]
It follows that
\[
 \frac{A_F(R)}{R^2\log R}
 \geq
 \frac{m^2}{M^2}\,
 \frac{\log(R/M)}{\log R}\,
 \frac{A_f(R/M)}{(R/M)^2\log(R/M)},
\]
which proves the final assertion.
\end{proof}

We can now deduce the angle-dependent existence theorem directly from the
Lagrangian case.  

\begin{theorem}	\label{t:weakiii}
If $\varepsilon > \cot \frac{\theta}{2}$, then there exists a proper holomorphic immersion $F:\D \to \C^2$ whose image is contained in $C_{\varepsilon} (\Pin)$ with 
\begin{align}
	\frac{A_F (R)}{R^2 \, \log R} \to \infty \, .
\end{align}
\end{theorem}

\begin{proof}
Suppose that
\[
 \varepsilon>\cot\frac{\theta}{2}=\frac{\sin\theta}{1-\cos\theta}\,  .
\]
We may choose
\[
 1<\eta<
 \frac{\varepsilon}{\sin\theta+\varepsilon\cos\theta}.
\]
This choice of $\eta$ implies both $\eta\cos\theta<1$ and 
 $\nu_\theta(\eta)<\varepsilon$.  
Apply Theorem~\ref{t:main} to the Lagrangian plane $\Pin_0$ with aperture
$\eta$, and then apply Lemma~\ref{l:transport}.  The resulting complete proper
holomorphic immersion is area-minimizing, lies in $C_\varepsilon(\Pin)$ and has area growth faster
than $R^2\log R$.  
\end{proof}

\section{Rapid growth holomorphic disks}	\label{s:rapid}

In this section we prove Theorem \ref{t:angleversion}(iii) by showing that if $\varepsilon>\cot(\theta/2)$, then   the cone 
$C_\varepsilon(\Pin)$ contains a complete proper holomorphic immersion of the
disk whose area grows arbitrarily fast. In the previous section, we proved the existence of superquadratic examples in this range. 
Now
 we show that the construction can be modified to get arbitrarily large area growth.

\begin{theorem}[Rapid area growth above the upper threshold]
\label{t:rapid}
Suppose that 
 $\varepsilon>\cot\frac\theta2$.  
For every nondecreasing function $\Phi:[1,\infty)\to(0,\infty)$ there is a
complete proper holomorphic immersion $F:\D\longrightarrow\C^2$  
and radii $R_j\to\infty$ such that
\[
 F(\D)\subset C_\varepsilon(\Pin)
 \qquad\text{and}\qquad
 A_F(R_j)\geq\Phi(R_j)
\]
for every $j$.  In particular, one may arrange that $A_F$ has no polynomial
upper bound.
\end{theorem}

The last assertion follows, for example, by taking $\Phi(R)=e^R$.  
``No polynomial upper bound'' means that there do not exist finite constants
$C$ and $d$ for which $A_F(R)\leq CR^d$ for all sufficiently large $R$.

The starting point is the holomorphic disk constructed in Theorem \ref{t:main}.  To this map,
 we add a high degree perturbation that is small in $C^0$ norm but carries large area (this large area is guaranteed by the high degree).  This procedure can then be repeated further and further out to get rapid area growth yet the map remains
  in the cone $C_{\varepsilon}(\Pi)$ as the $C^0$ difference is small.

\subsection{Polynomial perturbation}

We set $D_r=\{\zeta\in\C:|\zeta|<r\}$.  

\begin{lemma} 
\label{l:packet}
Let  $f$ be holomorphic on a neighborhood
of $\overline\D$. Given $M > 1$, $\beta > 0$ and  $0<r<1$, there exists a holomorphic polynomial   $q$ and 
$s\in(r,1)$ such that
\begin{enumerate}
\item $\|q\|_{C^0(\overline\D)} \leq \beta$;
\item $\|q\|_{C^1(\overline D_r)} \leq \beta$;
\item 
 $\operatorname{Area}_{f+q}(D_s)>M.$
\end{enumerate}
\end{lemma}

\begin{proof}
Fix a unit vector  $a \in \C^2$.
Let $N >2$ be a positive integer and $\delta > 0$ both to be chosen below and set
\[
 q_N(\zeta)=\delta \, \zeta^N a.
\]
It follows that (1) holds as long as $\delta \leq \beta$, no matter what $N$ is.  Moreover, we have
\[
 \|q_N\|_{C^1(\overline D_r)}
 \leq \delta r^N+\delta N r^{N-1} \leq \delta \, (N+1) \, r^{N-1} \, .
 \]
 Thus, for $r$ fixed, we get (2) as long as $N$ is sufficiently large.
Set $s_N=1-N^{-1}$ and compute
\[
 \|q_N'\|_{L^2(D_{s_N})}^2
 =\pi\delta^2N s_N^{2N} \, .
\]
In particular, $\|q_N'\|_{L^2(D_{s_N})}\to\infty$, because
$s_N^{2N}\to e^{-2}$. 
 The reverse triangle inequality gives
\[
 \bigl\|(f+q_N)'\bigr\|_{L^2(D_{s_N})}
 \geq
 \|q_N'\|_{L^2(D_{s_N})}-\|f'\|_{L^2(\D)}.
\]
The right-hand side tends to infinity as $N$ does.  This gives (3) since 
  a holomorphic map has area density equal to the
squared norm of its complex derivative.
\end{proof}

Notice that the image of the map remains in the fixed ball
\[
 B_{\sup_{\overline\D}|f|+\delta}(0),
\]
whereas its area tends to infinity with $N$.  This separation between
$C^0$ size and area is the essential point.

\subsection{Rapid growth in a Lagrangian cone}

We use the boundary-lifting construction in
\cite[Propositions~2.1 and~2.4 and Theorem~1.2]{FG} to produce surfaces with rapid growth.

\begin{proof}[Proof of Theorem \ref{t:rapid}]
It suffices to show this in the case of a Lagrangian two-plane and then use the linear map from Lemma \ref{l:transport}
to transport the statement to a general two-plane. 

 We describe how to  modify the proof of Theorem \ref{t:main} using
\cite[Theorem~1.2]{FG}.    We begin with the inductive step in the construction.  Define a sequence $\epsilon_k = \Lambda^{-k} \, \bar{\epsilon}$ where $\Lambda > 2$ will be chosen large below.  We will construct an  increasing sequence of radii $$r_1 < r_2 < \dots \to 1$$ and polynomial maps $f_k$  to $\C^2$ so that
  \begin{enumerate}
  \item $\rho \circ f_k > (1+a)^k$ for $|z| \in [r_k , 1]$.
  \item  $ \rho \circ f_k >   \rho \circ f_{k-1}- \epsilon_{k-1}$ for $|z| \leq 1$
\item $|f_k - f_{k-1}|_{C^1} < \epsilon_{k-1}$ for $|z| \leq r_{k-1}$.

  \item $|f_k| \leq b^k$ for $|z| \leq 1$.
  \end{enumerate}
  To explain this, we need to review the iterative construction in \cite{FG} that constructs $f_k$ and $r_k$ once we have $f_{k-1}$ and $r_{k-1}$.  Lemma $2.5$ in \cite{FG} gives $c>0$ so that for each $\e^{i\, \phi}$ the intersection of the Euclidean ball of radius $c\, \sqrt{\rho\circ f_{k-1} (\e^{i\phi})}$ with the Levi hypersurface through $f_{k-1} (\e^{i\phi})$ is 
  conformal to the disk.  This gives a continuous family of holomorphic maps $\tilde{\lambda} (\e^{i\phi}, w) \to \C^2$ with
  $\tilde{\lambda} (\e^{i\phi}, 0) = f_{k-1} (\e^{i\phi})$ so that  $\tilde{\lambda} (\e^{i\phi}, \cdot)$ maps the disk to the Levi hypersurface and its boundary to the boundary.  This gives (4). Moreover, 
it also implies that  $\rho$ goes strictly up by a factor on the boundary.  This is where (1) and (2) will come from.   Define 
  $$
  	 \lambda (\e^{i\phi}, w) =  \tilde{\lambda} (\e^{i\phi}, w) - f_{k-1} (\e^{i\phi}) \, , 
  $$
  to get a continuous family of holomorphic maps that sends $0$ to $0$.   The boundary lifting argument in Proposition $2.1$ in \cite{FG}  then chooses holomorphic polynomials $A_j$ to get an approximation
  $$
  	 \lambda (\e^{i\phi}, w) \approx  \bar{\lambda} (\e^{i\phi}, w) \equiv \frac{w}{\e^{M\,i\phi}} \, \sum_{j=1}^N A_j (\e^{i\phi}) \, w^{j-1} \, , 
  $$
  where the series starts at $j=1$ because $0$ is sent to $0$.  The map $f_k$ is given by choosing a large positive integer $\ell$ and setting
  $$
  	f_k (w) = f_{k-1} (w) +  \bar{\lambda} (w, w^{\ell} ) = f_{k-1}(w) + w^{\ell - M} \, \sum_{j=1}^N A_j (w) \, w^{(j-1)\, \ell}  \, .
  $$
  Taking $\ell$ large makes the first exponent $\ell -M$ positive and large, which kills off $f_k - f_{k-1}$ when $|w|$ is strictly  less than one (e.g., when $|w| \leq r_{k-1}$).   This gives the $C^0$ estimate in (3).  The $C^1$ estimate follows by differentiating the series; this introduces a factor of $\ell$, but it still goes to zero for $|w| \leq r_{k-1}$ as $\ell \to \infty$ because the powers kill off the linear factor.  Finally, we have already observed that (1) holds at the boundary and $f_k$ is a polynomial, so there is some radius $r_k < 1$ so that it holds in the annular region as claimed. 
  
The bound (2) is given by (ii) in Proposition $2.4$ in \cite{FG} (cf. (ii) in Proposition $2.1$ in \cite{FG} and $(b)$ in Theorem $3.1$ in \cite{FG}).   Properness follows since (1), (2) and (3) force $\rho \circ f_k$ to go to infinity as $z \to 1$ and $k$ increases.  
    
 It remains to explain how to get the area lower bound.  At the $k$-th stage
  we alter the map given by \cite{FG} using Lemma~\ref{l:packet} with
  $r=r_{k-1}$, with $\beta$ so small that (1), (2) and (4) still hold, and with
  area $M\ge\bigl(\Phi(b^k+1)^{1/2}+2\sqrt\pi\bigr)^2$ on $\D_s$ for some
  $s\in[r_k,1)$; we make this $s$ the new $r_k$.  By (3), $|F-f_k|\le2\epsilon_k$
  and $|F'-f_k'|\le2\epsilon_k$ on $\D_{r_k}$.  Hence $F(\D_{r_k})\subset B_{R_k}$
  with $R_k=b^k+1$, and the reverse triangle inequality in $L^2$ gives
  $\Area_F(\D_{r_k})^{1/2}\ge M^{1/2}-2\sqrt\pi\,\epsilon_k\ge\Phi(R_k)^{1/2}$, so
  $A_F(R_k)\ge\Phi(R_k)$.

Perturbing the examples as in the proof of Theorem \ref{t:main} gives an immersion.  Finally, the case for  general 
$\theta$ follows as before by transporting the construction over a Lagrangian plane via a linear map.   
\end{proof}

\section{The algebraic threshold}\label{s:algebraicity}

In this section, we will prove the doubling theorem, Theorem \ref{t:doubling}, and algebraicity    as a consequence. 

Throughout this section,  $\Pin$ is a two-plane in $(\R^4,J)$ with K\"ahler angle $\theta$ and $\Sigma$ is a connected open
Riemann surface.

\subsection{Doubling}

We prove a local doubling bound for complex maps whose image locally lies in a cone with aperture 
below the   threshold $\cot\frac\theta2$.    The key will be  that these cones
 miss a  particular complex line. 
 This allows us to produce
  a complex-linear function that is proper of finite degree on any curve contained in the cone.  
 In the global case, where the entire image is in the cone, this
   forces the curve to be algebraic.
  
 \begin{figure}[ht]
\centering
\begin{tikzpicture}[
   x=1.15cm,
   y=1.15cm,
   scale=.7,
   line cap=round,
   line join=round
]
   \fill[blue!9]
     (0,0)--(4.2,1.15)--(4.2,-1.15)--cycle;
   \fill[blue!9]
     (0,0)--(-4.2,1.15)--(-4.2,-1.15)--cycle;

   \draw[thick,blue!65!black]
     (-4.2,-1.15)--(4.2,1.15);
   \draw[thick,blue!65!black]
     (-4.2,1.15)--(4.2,-1.15);

   \draw[very thick]
     (-4.35,0)--(4.35,0)
     node[right] {$\Pin$};

   \draw[thick,densely dashed]
     (-4.1,-0.43)--(4.1,0.43)
     node[right] {$L_-$};

   \draw[thick,red!70!black]
     (0.4,-2.65)--(-0.4,2.65)
     node[above right] {$L_+$};

   \fill (0,0) circle (1.3pt);

   \node[blue!65!black,above, right]
     at (3.05,-0.77) {$C_\varepsilon(\Pin)$};
\end{tikzpicture}
\caption{In the
algebraic regime
$\tan(\theta/2)\leq\varepsilon<\cot(\theta/2)$, the cone $C_\varepsilon(\Pin)$ contains
$L_-$,
but intersects $L_+$ only at $0$.  Thus in complex coordinates $(z_+,z_-)\in L_+\oplus L_-=\C^2$,
 the second coordinate $z_-$ is proper on any proper holomorphic curve
in $C_\varepsilon(\Pin)$. This is the key for showing 
quadratic area growth.}
\label{f:quadratic-regime}
\end{figure}
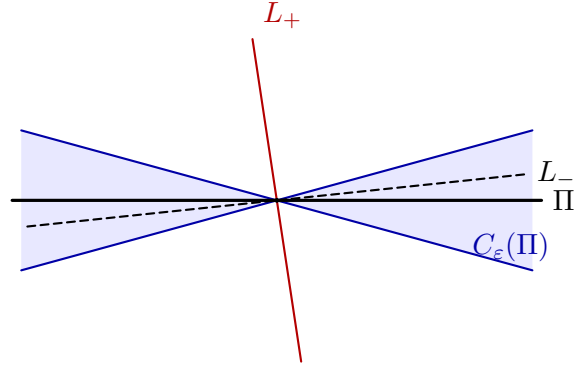

Corollary
\ref{c:lines}(ii) and (iii) gives orthogonal complex lines $L_+$, $L_-$ with  $L_-\oplus L_+=\C^2$, where
$L_{\pm}$ has an orthonormal frame  $\{ v_{\pm} , J \, v_{\pm} \}$ (see Figure \ref{f:quadratic-regime}).
 Each vector  $z\in \C^2$, has a unique orthogonal decomposition
$$
	z= x_- \, v_- + y_- \, J \, v_- + x_+ \, v_+ + y_+ \, J \, v_+ \, , 
$$
where $x_- \, v_- + y_- \, J \, v_-$ is the projection to $L_-$ and $x_+ \, v_+ + y_+ \, J \, v_+$ is the projection to $L_+$.
Using this, define complex coordinates $z_- = x_- + i \, y_-$ and $z_+ = x_+ + i \, y_+$. Then $z_-$ and $z_+$
 are linear holomorphic maps from $\C^2$ to $\C$.

\begin{lemma}[Properties of $z_-$]   \label{l:zminus}
 If
$\tan \frac{\theta}{2} \leq \varepsilon < \varepsilon' < \cot \frac{\theta}{2}$, then
\begin{enumerate}
\item	There exists $c = c(\varepsilon' , \theta) > 0$ so that   
\begin{equation}\label{e:lam1}
	 c\,|x| \le |z_-(x)| \qquad\text{for all }x\in C_{\varepsilon'}(\Pin).
\end{equation}
\item Given $K$, there exists $R_0 = R_0 (\varepsilon , \varepsilon' , K)$   so that
  $S_{\varepsilon,K}(\Pi)\setminus B_{R_0}\subset C_{\varepsilon'}(\Pi)$.
\end{enumerate}
\end{lemma}

\begin{proof}
We have that  $L_+\cap C_{\varepsilon'}(\Pin)=\{0\}$ and $L_-\subset C_{\varepsilon}(\Pin)$.
 Let $\partial B_1$ be the unit sphere in $\R^4=\C^2$.  Since
$z_-$ vanishes precisely on $L_+$, the sets
$L_+\cap \partial B_1$ and $C_{\varepsilon'}(\Pin)\cap \partial B_1$ are disjoint compact sets, so $|z_-|\ge c>0$
on $C_{\varepsilon'}(\Pin)\cap \partial B_1$ for some constant $c>0$. 
Combining this with  homogeneity gives the first claim.

  To show the second claim,   suppose that 
$|P^\perp x|^2\le (\varepsilon|Px|+K)^2\le 2\,\varepsilon^2\,|Px|^2+2\,K^2$ and $|x|\geq R_0$.  It follows that $$R_0^2\le |x|^2=|Px|^2+|P^{\perp}x|^2\le (1+2\,\varepsilon^2)\,|Px|^2+2\,K^2$$
 and therefore $1\leq \frac{1+2\,\varepsilon^2}{R_0^2-2\,K^2}\,|Px|^2$.     Since $\varepsilon'>\varepsilon$,
 it follows easily that if $R_0$ is sufficiently large, then $|P^{\perp}x|\le \varepsilon'\,|Px|$ as claimed.  
\end{proof}

The proof of quadratic area growth will use the next lemma:

\begin{lemma}\label{l:twofcts}
Suppose that $\gamma$ and $\eta$ are functions on a manifold $M$ with   
$|\gamma|\leq C(1+\eta)$ for a constant $C>0$, where $\gamma$ is harmonic, $\eta\geq 0$ is  locally Lipschitz,
and $\{ \eta < 2\, R\}$ is compact and disjoint from $\partial M$ for some $R > 1$, then 
$$
\int_{\{\eta<R\}}|\nabla \gamma|^2
\leq 36\,C^2\int_{\{\eta<2\,R\}}|\nabla \eta|^2\,  .
$$
\end{lemma}

\begin{proof} 
Choose a Lipschitz cutoff
$$
\phi_R=\chi(\eta/R),
$$
where \(\chi=1\) on \([0,1]\), vanishes on \([2,\infty)\), and \(|\chi'|\leq1\).  Since 
$\{ \eta < 2\, R\}$ is compact and disjoint from $\partial M$, $\phi_R$ has compact support on $M$. 
The chain rule gives
$$
|\nabla\phi_R|\leq \frac{1}{R}\,|\nabla \eta|.
$$
Since \(\gamma\) is harmonic and $\phi_R$ has compact support, the Caccioppoli inequality gives
$$
\int_{\{\eta<R\}}|\nabla \gamma|^2\leq \int_{\{\eta<2\,R\}}|\nabla \gamma|^2\,\phi^2_R
\leq 4\int_{\{R<\eta<2\,R\}} \gamma^2|\nabla\phi_R|^2\le \frac{4}{R^2}\int_{\{R<\eta<2\,R\}} \gamma^2\,|\nabla \eta|^2\,  .
$$
On \(1<R\leq \eta\leq 2\,R\), we have that \(|\gamma|\leq C\,(2\,R+1)\le 3\,C\,R\) and hence the claim follows.  
\end{proof}
\vskip1mm
We now turn to the proof of the doubling theorem that guarantees quadratic area growth below the algebraic threshold 
 $\cot \frac{\theta}{2}$.

\begin{proof}[Proof of Theorem \ref{t:doubling}]
If $\varepsilon < \tan \frac{\theta}{2}$, then $C_{\varepsilon}(\Pin) \subset C_{ \tan \frac{\theta}{2}}(\Pin)$,
so we can assume that  
$\tan \frac{\theta}{2} \leq \varepsilon$.  Choose $\varepsilon'$ so that 
$\tan \frac{\theta}{2} \leq \varepsilon < \varepsilon' < \cot \frac{\theta}{2}$. Lemma \ref{l:zminus} gives $c>0$ and $R_0$ so that
\begin{align}\label{e:lam1a}
	 c\,|z| &\le |z_-| \quad\text{ for all } z\in  B_{\Omega \, R} \cap F(\Sigma) \setminus B_{R_0}
	 \, .
 \end{align}
In the rest of this proof, we work on $F^{-1} (B_{\Omega \, R})$. 
 Set $h = z_- \circ F$ and observe that $\log |h|$ is harmonic since $\log |z|$ is harmonic on $\C$ and $h$ is a holomorphic map. 
As long as  $r$ satisfies $R_0 \leq r < \frac{ \Omega \, R}{\sqrt{ 1+ c^{-2}}}$, (\ref{e:lam1a})  gives that the image of $\{ |h| = r \}$ is outside of $B_{R_0}$ and strictly inside of $B_{\Omega \, R}$.  Thus, if we take $R_1 \geq R_0$ and $\Omega > 2\, \sqrt{ 1+ c^{-2}}$, then for every
$r \in [R_1 , 2R]$, 
the set where $R_1 \leq |h| \leq r$ is compact and has $|h|$ constant along the boundary. In fact, $|h|$ is $R_1$ on the inner boundary and $r$ on the outer boundary.
 Thus, given any $r \in [ R_1 , 2R]$, the divergence theorem and the harmonicity of $\log |h|$ give that 
 \begin{align}		\label{e:e36}
 	r^{-1} \, \int_{ |h| = r} |\nabla |h|| =  \int_{ |h| = r} |\nabla \log  |h|| = \int_{|h| = R_1} |\nabla \log |h|| \equiv I_1 \, .
\end{align}
It follows from this and the co-area formula that
\begin{align}	\label{e:CAF}
	\int_{ R_1 \le  |h| \le R} |\nabla |h||^2 = \int_{R_1}^R \,  \left( \int_{ |h| = r} |\nabla |h|| \right) \, dr = I_1 \, 
	\int_{R_1}^R r \, dr = \frac{I_1}{2} \, \left( R^2 - R_1^2 \right) \, .
\end{align}
We apply this first with $R=2\, R_1$ to bound $I_1$ in terms of area
\begin{align}	\label{e:littleom}
	\frac{3}{2} \, R_1^2 \, I_1 = \int_{R_1 \leq |h| \leq 2 \, R_1} |\nabla |h||^2 \leq \Area_F (R_1 \leq |h| \leq 2\, R_1)
	\leq A_F ( 2\, R_1/c) \, .
\end{align}
Set $\omega = \frac{2}{c}$.  Next, we apply (\ref{e:CAF}) with up to radius $2\, R$ to get that
\begin{align}	\label{e:CAF2}
	\int_{ R_1 \le  |h| \le 2\, R} |\nabla |h||^2 \leq 2\, I_1 \, R^2 \le   \frac{4}{3} \, \left( \frac{R}{R_1} \right)^2 \, A_F (\omega \, R_1) \, , 
\end{align}
where the last inequality is (\ref{e:littleom}).

For every coordinate function $x_i$, we have that  $x_i \circ F$ is harmonic on $\Sigma$ and is bounded by a multiple of $h$ plus a constant 
(the constant is needed when the image lands inside of $B_{R_0}$).   Thus, we can apply Lemma \ref{l:twofcts} to each coordinate function to get
\begin{align}
	\int_{  |h| < R} |\nabla x_i \circ F|^2 \leq  C' \, \int_{ |h| < 2\, R}  |\nabla |h||^2 \leq C \, R^2 \, A_F (\omega \, R_1) \, ,
\end{align}
where the last inequality used (\ref{e:CAF2}) to bound the area in the annulus $\{ R_1 < |h| < 2\, R\}$
and the trivial area bound in $\{ |h| < R_1 \}$. 
Finally, summing this over $i$ gives a bound for the energy of $F$ and thus the area
\begin{align}
	\Area_F (|h| < R) \leq C \, R^2 \, A_F (\omega \, R_1) \, .
\end{align}
This gives the claim since $B_R \subset \{ |z_-| < R \}$. 
\end{proof}

\vskip2mm

\begin{proof}[Proof of (ii) in Theorem \ref{t:angleversion} ]
Existence follows from (ii) in Corollary \ref{c:lines}. 
As in the proof of Theorem \ref{t:doubling}, set $h = z_- \circ F$  so that $\log |h|$ is a proper harmonic function outside of a compact set that goes to infinity at infinity.  The existence of such a function is the Khas'minskii criterion for parabolicity, \cite{Kam}.
Theorem \ref{t:doubling} gives quadratic area growth.   Since $F$ is proper, the proper mapping theorem
(Theorem 8.8 in \cite{Dem})
 implies that
$V=F(\Sigma)$ is a closed one-dimensional analytic subset of $\C^2$. Stoll's algebraicity theorem \cite{Sdim1,S} therefore implies that $V$
is an affine algebraic curve.  
\end{proof}

\vskip1mm
There is also a useful consequence for stable minimal surfaces.  Micallef proved that
every complete oriented parabolic stable minimal surface in $\R^4$ is holomorphic for some
parallel orthogonal complex structure \cite{Mi}; cf. \cite{AS1,AS2}.  Combining his theorem with part (ii) in Theorem
\ref{t:angleversion} gives:

\begin{corollary}\label{c:micallef}
Let $\Sigma\subset\R^4$ be a complete oriented properly immersed stable minimal surface without
boundary, of parabolic conformal type.  Suppose that it is contained in
$S_{\varepsilon,K}(\Pin)$ for some two-plane $\Pin$, $K<\infty$, and $\varepsilon<1$.
Then there is a parallel orthogonal complex structure $\overline J$ on $\R^4$ for which the image
of $\Sigma$ is an affine algebraic curve and $A(R)\leq CR^2$.  Moreover, writing
$\theta_{\overline J}(\Pin)$ for the K\"ahler angle of $\Pin$, we have
\[
 \theta_{\overline J}(\Pin)\leq2\arctan\varepsilon<\frac\pi2.
\]
\end{corollary}

In particular, $\Pin$ is not Lagrangian for $\overline J$.  Hence a stable oriented counterexample
to the narrow-cone question in $\R^4$ must be nonparabolic; Lemma
\ref{l:nonparabolic-growth} then gives $A(R)/(R^2\log R)\to\infty$.

\begin{proof}
By Micallef's theorem \cite{Mi} a complete oriented parabolic stable minimal immersion into $\R^4$
is holomorphic with respect to some parallel (i.e.\ constant) orthogonal complex structure
$\overline J$ on $\R^4$.  Properness of the immersion is exactly properness of the resulting
holomorphic map.  Corollary \ref{p:nonexist}, applied with $J=\overline J$, gives
\[
 \varepsilon\geq\tan\frac{\theta_{\overline J}(\Pin)}2,
\]
and hence $\theta_{\overline J}(\Pin)\leq2\arctan\varepsilon<\pi/2$.  Part (ii) in Theorem \ref{t:angleversion}
gives the algebraicity and area conclusions.   
\end{proof}

\section{Narrow real cones and polynomial growth}\label{s:narrow}

The confined volume-doubling theorem of \cite{CM1} gives a polynomial bound in every dimension and codimension as long as
the stationary varifold $\Sigma$ is in a sufficiently narrow cone.    Moreover, \cite{CM1} gives Euclidean volume growth when the
height grows sublinearly.  We are first interested here in what happens when the cone aperture is very small.  Theorem \ref{t:realsetting}, part (ii), shows that  as the aperture decreases, the doubling improves and the growth gets as close as we want to Euclidean.  Second, we construct  elementary examples showing that polynomial growth can fail for large cone angles.

   Let $\Pin$ be an $n$-plane in $\R^{n+k}$ and
write $P$ and $P^\perp$ for the corresponding orthogonal projections. Let $\Sigma^n\subset\R^{n+k}$ be a stationary integral varifold, set
\[
 s(p)=|P(p)|, \quad \Sigma_{s_1 , s_2} = \Sigma \cap \{ s_1 < s < s_2 \} \, 
 \qquad V(R)=\operatorname{Vol}_n\bigl(\{ s  <R\} \cap \Sigma \bigr) \, .
\]
 Let  $\kappa_n>0$ be the doubling threshold given by Theorem $0.5$ in \cite{CM1}.

\begin{lemma}	\label{l:cm1D}
There exist $c=c(n) , c' = c'(n)$  so that  if $\varepsilon < \frac{\kappa_n}{4}$,  $\omega >2$
and $\Sigma_{\frac{1}{4} , 4\, \omega} \subset C_{\varepsilon}(\Pin)$, then 
\begin{align}	\label{e:cm1C}
	V(\omega) \leq c \, \omega^n \, V(1) + c' \, \varepsilon^2 \, V(\omega) \, \log \omega \, .
\end{align}
\end{lemma}

\begin{proof}
It will be convenient to let $E = E(x)= \frac{1}{n} \, \sum_{j=1}^{n} |P^{\perp} v_j|^2$ denote the pointwise tilt excess relative to the $n$-plane $\Pin$ where the $v_j$'s are an orthonormal basis for $T_x\Sigma$.
We can assume that $\omega=2^N$. 
We will need three results from \cite{CM1}.  First, Lemma $2.5$ in \cite{CM1} gives that
\begin{align}	\label{e:cm1A}
	V(\omega) \leq \left( 1 + n^2 \, 2^n \, \omega^n \right) \, V(1) + n^2 \, \omega^n \, \int_{\Sigma_{1,\omega}} \frac{E}{s^n} \, . 
\end{align}
The next result that we need combines Lemma $2.18$ in \cite{CM1} with the cone condition to get 
\begin{align}	\label{e:cm2B}
	n\, \int_{\Sigma_{2^i, 2^{i+1}}} E &\leq \frac{16}{ 2^{2\, i}} \, 	\int_{\Sigma_{2^{i-1}, 2^{i}}}  \left| P^{\perp} x\right|^2 +
	 \frac{4}{ 2^{2\, i+2}} \, 	\int_{\Sigma_{2^{i+1}, 2^{i+2}}}  \left| P^{\perp} x\right|^2 \notag \\
	 &\leq 16 \, \varepsilon^2 \, \left( V(2^i) + V(2^{i+2}) \right) \leq 32 \, \varepsilon^2 \,  V(2^{i+2}) \, .
\end{align}
Monotonicity gives $V(2^{i+2}) \leq V(2\omega) \, \left( \frac{2^{i+2}}{2\omega} \right)^n \, (1+\varepsilon^2)^{ \frac{n}{2}}$ for $i \leq N-1$, so (\ref{e:cm2B})
 implies 
\begin{align}	\label{e:cm2C}
	n\, \int_{\Sigma_{2^i, 2^{i+1}}} \frac{E}{s^n} &\leq  32 \, \varepsilon^2 \, 2^{-i \,n}  \,  V(2^{i+2})
	\leq 32 \, \varepsilon^2 \,   2^{-i \,n} \, V(2\omega) \, \left( \frac{2^{i+2}}{2\omega} \right)^n \,  (1+\varepsilon^2)^{ \frac{n}{2}} \notag \\
	&=
	2^{n+5} \, \varepsilon^2  \, \frac{V(2\omega)}{ \omega^{n}} \, (1+\varepsilon^2)^{ \frac{n}{2}}	 \, .
\end{align}
Summing this from $i=0$ to $N-1$ and using this in (\ref{e:cm1A})
gives that 
\begin{align}	\label{e:cm1AA}
	V(\omega) \leq \left( 1 + n^2 \, 2^n \, \omega^n \right) \, V(1) + n \, 	2^{n+5} \, \varepsilon^2 \, N \,  (1+\varepsilon^2)^{ \frac{n}{2}}  \, V(2\omega) \, . 
\end{align}
The volume doubling bound for $\frac{V(2\omega)}{V(\omega)}$ from Theorem $0.5$ in \cite{CM1} completes the proof.
\end{proof}

We will use the lemma to prove a volume doubling
 by choosing $\omega$ small enough that the last term in (\ref{e:cm1C}) can be absorbed on the right-hand side.  Iterating this doubling 
 will give
 the polynomial volume growth.

\begin{proof}[Proof of   Theorem \ref{t:realsetting}]
Part (i) is proven in Theorem $0.4$ of \cite{CM1}. Theorem $0.5$ in \cite{CM1} gives $\kappa_n > 0$ so that the growth is polynomial below $\kappa_n$.   Theorem \ref{t:catenoid-sum}  gives super-polynomial examples in every codimension when $\varepsilon\ge\varepsilon_*\approx0.66$. Thus, there exists some maximal $\delta_n > 0$ so that the growth is polynomial below $\delta_n$ and super-polynomial above it. 
It remains to prove the dependence on $\varepsilon$ in (ii); we do this next.

We will first assume that $\Sigma  \subset C_{\varepsilon}(\Pin) \cup B_{\frac{1}{4}}$.
Let $c$ and $c'$, both depending just on $n$, be given by Lemma \ref{l:cm1D}. 
Define $\Omega > 1$ and $\delta > 0$ by 
\begin{align}
	\Omega = \left(2\, c \right)^{ \frac{1}{ \delta}} {\text{ and }}
	\varepsilon^2 = \left( 2 \, c' \,   \log \Omega \right)^{-1} \, . \label{e:consts}
\end{align}
Set $\delta_n' = \min \{ \kappa_n/4 , (2\, c'\, \log 2)^{-1/2} \}$. 
If  $\varepsilon < \delta_n'$, then $\Omega > 2$ as required. 
We have that
\begin{align}	\label{e:deltaIS}
	\delta = \frac{ \log (2\, c)}{\log \Omega} = 2\, c' \, \log (2\, c)  \, \varepsilon^2 = c_n \, \varepsilon^2 \, ,
\end{align}
where the last equality defines $c_n$.
 Lemma \ref{l:cm1D}  and (\ref{e:consts}) give   that
 \begin{align}
 	V( \Omega) &\leq c \, \Omega^n \, V(1) + c' \, \varepsilon^2 \, \log \Omega \, V(\Omega) = c \, \Omega^n \, V(1) + \frac{1}{2} \, V(\Omega) \, .
 \end{align}
   Thus, we see   that 
 \begin{align}
 	V(\Omega) \leq 2\, c \, \Omega^n \, V(1) = \Omega^{n+\delta} \, V(1) \, .
\end{align}
By scaling this estimate and applying it repeatedly, we see for every $m$ that 
\begin{align}
	V(\Omega^m) \leq  \Omega^{n+\delta} \, V(\Omega^{m-1})\leq \cdots \leq \left(  \Omega^{n+\delta} \, \right)^m \, V(1) 
	= \left( \Omega^m \right)^{n+\delta} \, V(1) 
	 \, .
\end{align}
Since $V(r)$  is increasing, we get the same polynomial dependence for any $r$, not just powers of $\Omega$, but with a worse constant.
Since $\delta = c_n \, \varepsilon^2$ by  (\ref{e:deltaIS}), this gives the dependence in (ii) when $K=0$. 
The case with nonzero $K$ follows 
by rescaling so that $\Sigma \subset C_{2\varepsilon}(\Pin) \cup B_{\frac{1}{4}}$ and then replacing $c_n$ by $4\, c_n$ and
$\delta_n'$ by $\delta_n'/2$.
\end{proof}

We show that there is an aperture, and give an explicit bound for it, so that above this there are 
varifolds with arbitrary volume growth.  
The catenoid $\Sigma = \{ \cosh^2 x_3 = x_1^2 + x_2^2\}$  is a  rotationally symmetric minimal surface in $\R^3$
that 
 does not intersect the cylinder $\{x_1^2 + x_2^2 < 1\}$ (see subsection \ref{ss:cat}).  Let $\Pin=\{x_3=0\}\subset\R^3$ and let
\[
  \varepsilon_*=\max_{t\ge0}\frac{t}{\cosh t}=\frac1{\sinh t_*}\approx0.6627,
  \qquad\text{where } t_*\tanh t_*=1 .
\]
It follows that 
\begin{align}	\label{e:catconea}
	\Sigma \subset C_{\varepsilon *}(\Pin)  \, .
\end{align}
Given 
 $a>0$,  let
$\Sigma_a = \{ a\, x\, | \, x \in \Sigma\}$ be the scaled catenoid with neck radius $a$.
  It follows that $\Sigma_a \cap B_a = \emptyset$ 
and $\Sigma_a$ satisfies (\ref{e:catconea}) since cones are preserved by scaling.

\begin{theorem}[Catenoid sums]\label{t:catenoid-sum}
For every $\varepsilon\ge\varepsilon_*$ and every nondecreasing function
$\Phi:[1,\infty)\to(0,\infty)$, there is a stationary integral $2$-varifold
$V$ in $\R^3$ with $\spt\|V\|\subset C_\varepsilon(\Pin)$ and
$\|V\|(B_R)\ge\Phi(R)$ for all $R\ge4$.
\end{theorem}

\begin{proof}
The varifold $V$ will be given by sums of scaled catenoids $m_i \, \Sigma_{R_i}$ with multiplicities $m_i$ and increasing scales 
$R_i \to \infty$.
 Choosing $m_i$ large enough gives as much mass as we want.   Every $\Sigma_{R_i}$ lies in the cone
and only finitely many of them intersect each compact set.
\end{proof}

These catenoids intersect each other, so the varifold $V$ is not smooth and embedded.  Moreover, $V$ cannot be parameterized as an immersion from a single smooth connected surface.  This contrasts with the earlier examples of rapid area growth in Theorem \ref{t:angleversion}(iii).

\subsection{A holomorphic analog of the catenoid}
Write $z_j=x_j+i\,y_j$ and let $\Pin_0=\{y_1=y_2=0\}$ be the Lagrangian plane
spanned by the $x_1$- and $x_2$-axes, so that $|Pz|^2=x_1^2+x_2^2$ and
$|P^\perp z|^2=y_1^2+y_2^2$.  Taking real and imaginary parts, the conic
$Q=\{z_1^2+z_2^2=1\}$ is given by
\[
  x_1^2+x_2^2=1+y_1^2+y_2^2 \qquad\text{and}\qquad x_1\,y_1+x_2\,y_2=0\, .
\]
The first equation says that $|P^\perp z|^2=|Pz|^2-1<|Pz|^2$ on $Q$, so
$Q\subset C_1(\Pin_0)$; since $|P^\perp z|/|Pz|\to1$ as $|z|\to\infty$ along $Q$,
no cone of smaller aperture contains $Q$, even outside a compact set.  The cone
is invariant under dilations, so each dilate $\lambda Q=\{z_1^2+z_2^2=\lambda^2\}$,
$\lambda>0$, also lies in $C_1(\Pin_0)$, and distinct dilates are disjoint.
Finally, $|z|^2=1+2|P^\perp z|^2\geq1$ on $Q$, so $Q$ does not meet the open unit
ball and touches the unit sphere exactly along the circle $Q\cap\Pin_0$; hence
$\lambda Q$ does not meet the open ball of radius $\lambda$.

The conic plays the role of the catenoid $\{r=\cosh x_3\}\subset\R^3$.  Both are
complete embedded minimal surfaces with two ends.
Both are invariant under rotations of the plane: the real rotations
$(z_1,z_2)\mapsto(z_1\cos t-z_2\sin t,\,z_1\sin t+z_2\cos t)$ preserve $\Pin_0$
and $Q$.  In both, the points closest to the origin form a neck
circle in the plane, here $Q\cap\Pin_0$.  They differ at infinity.  The catenoid
grows only logarithmically, so its blow-down is the plane with multiplicity two,
and the aperture $\max_t t/\cosh t\approx0.66$ of the smallest cone about
$\{x_3=0\}$ containing it is attained on two circles near the neck.  The
blow-down of $Q$ is the pair of complex lines $\{z_1=\pm i\,z_2\}$, which lie on
the boundary of $C_1(\Pin_0)$, so for $Q$ the aperture $1$ is approached only at
infinity.  Finally, scaled catenoids intersect one another, whereas distinct
dilates of $Q$ are disjoint.

As for the catenoid in Theorem~\ref{t:catenoid-sum}, one can look at sums of
dilates of $Q$: if $\lambda_j\to\infty$, then taking each $\lambda_jQ$ with
multiplicity $m_j$ gives a proper holomorphic curve in $C_1(\Pin_0)$, since
$\lambda_jQ$ misses $B_{\lambda_j}$, and choosing the $m_j$ large makes its area
grow as fast as we like.  Unlike scaled catenoids, distinct dilates of $Q$ are
disjoint, so the multiplicities can even be replaced by nearby distinct dilates
to give embedded examples.

\section{Stationary varifolds with growth
  \texorpdfstring{$r^{n+\frac{4}{5}\,\varepsilon^2}$}{r\textasciicircum(n+4/5 epsilon\textasciicircum 2)}}
\label{s:varifolds}

We prove (iv) in the introduction, showing that Theorem~\ref{t:realsetting}(ii) is
sharp: Euclidean growth fails in every narrow cone, even in
codimension one, and the excess exponent can be of order $\varepsilon^2$.
We work in $\R^3$ with $r=(x_1^2+x_2^2)^{1/2}$ and $\Pi=\{x_3=0\}$.
For a smooth embedded surface
$\Sigma$, possibly with boundary, $[\Sigma]$ is the multiplicity-one
varifold; $\|V\|$ is the mass of a varifold $V$.

\begin{theorem}\label{t:varifold}
If $\eta \in (0,1/10)$, then there is a
stationary integral $2$-varifold $V$ in $\R^3$ with:
\begin{enumerate}[label=\textup{(\roman*)}]
\item $\spt\|V\|\subset S_{\varepsilon,\eta}(\Pi)$ with   $\varepsilon=\tfrac{11}{10}\eta$. 
\item There exists $R_0<\infty$ such that
$\|V\|(B_R)\ge  R^{2+\frac{4}{5} \, \varepsilon^2}$ for all $R\ge R_0$.
\item $V=\sum_i m_i[\Sigma_i]$ is a locally finite sum with positive integer
multiplicities $m_i$; each $\Sigma_i$ is a flat disk, a flat annulus, or a
piece of a catenoid whose axis is the $x_3$-axis.
\item $V$ is invariant under rotation about the $x_3$-axis and reflection about $x_3 = 0$.
\end{enumerate}
\end{theorem}

Given $\varepsilon\in(0,1/10)$, Theorem~\ref{t:varifold} with
$\eta=\frac{10}{11}\varepsilon$ gives (iv), since
$S_{\varepsilon,\eta}(\Pi)\subset S_{\varepsilon,1}(\Pi)$.

Rotationally symmetric sets are
described by their profiles in the half-plane $\cH=\{(r,h):r>0\}$: for
$Y\subset\cH$ let $\Rot Y=\{(r\cos\phi,r\sin\phi,h):(r,h)\in Y,\ \phi\in\R\}$.

\subsection*{Idea of the construction}
The example $V$ is built from pieces of catenoids and two flat disks, each
with a positive integer multiplicity, all meeting along circles.  $V$ is
stationary when, at each circle, the unit tangent vectors of the
pieces meeting there, pointing into the pieces and counted with
multiplicity, sum to zero; this is one horizontal and one vertical
condition.  We start with two horizontal disks of large multiplicity and
attach to their boundary circles outgoing pieces of catenoids, one going up
and one going down, with multiplicities chosen to make this sum almost zero.
Because multiplicities are integers, the sum cannot be made exactly zero
this way; the small remaining error is cancelled by two bundles of nearly
flat catenoidal ends, whose angles can be adjusted continuously.

The construction then proceeds outward, from one pair of circles to the
next, twice as far out.  Along a catenoid the vertical flux
$r\sin(\text{slope angle})$ is constant, so each piece arrives at the next
circle flatter than it left.  A flatter sheet pulls harder horizontally on
the circle, so more sheets must leave the circle than arrive at it to
balance them.  The number of sheets therefore grows by a factor of at least
$\mu=1+0.7\,\eta^2$ from one circle to the next.  Since each sheet over
$\{2^k\le r\le2^{k+1}\}$ has area comparable to $4^k$, the area there is
about $(4\mu)^k=(2^k)^{2+\gamma}$ with $\gamma=\log_2\mu>\eta^2$: the area
grows like $r^{2+\gamma}$ rather than quadratically.

In more detail, see Figures~\ref{f:network} and \ref{f:two-catenoids}: $V$ is invariant under all
rotations about the $x_3$-axis and under the reflection $x_3\mapsto-x_3$, and
all pieces meet along the junction circles $\Rot\{P_k^\pm\}$,
$P_k^\pm=2^k(1,\pm\eta)$.  Between consecutive junctions we use catenoidal
arcs $S_k$ from $P_k^+$ to $P_{k+1}^+$ and $C_k$ from $P_k^+$ to
$P_{k+1}^-$, with multiplicities $s_k$ and $c_k$, and their reflections; the
ends have a fixed multiplicity $N$.  The horizontal part of the balance gives
the growth (Lemma~\ref{l:mult}); the other component keeps $s_k$ and
$c_k$ positive.

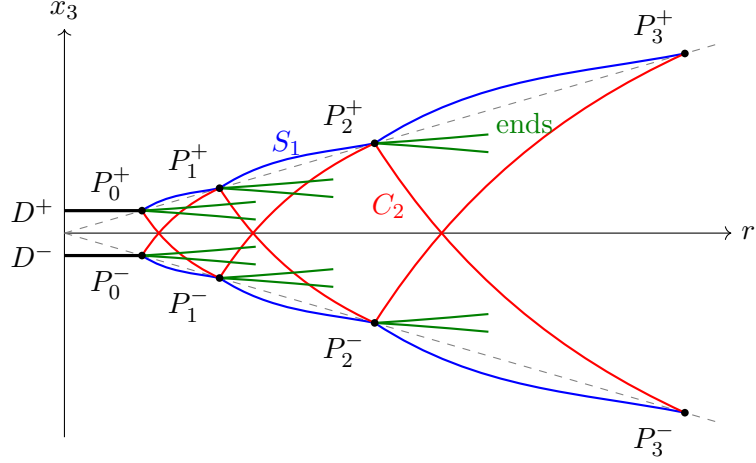
\begin{figure}[ht]
\centering
\begin{tikzpicture}[scale=.72,x=0.95cm,y=1.25cm]
  \draw[->] (0,0) -- (12.9,0) node[right] {$r$};
  \draw[->] (0,-3.0) -- (0,3.0) node[above] {$x_3$};
  \draw[dashed,gray] (0,0) -- (12.6,2.772);
  \draw[dashed,gray] (0,0) -- (12.6,-2.772);
  \draw[very thick] (0,0.33) -- (1.5,0.33);
  \draw[very thick] (0,-0.33) -- (1.5,-0.33);
  \node[left] at (0,0.33) {$D^+$};
  \node[left] at (0,-0.33) {$D^-$};
  \foreach \x in {1.5,3,6} {
    \pgfmathsetmacro{\y}{0.22*\x}
    \draw[thick,blue] (\x,\y) .. controls ({\x+\x/3},{\y+\x/6}) and ({2*\x-\x/3},{2*\y-0.05*\x}) .. ({2*\x},{2*\y});
    \draw[thick,blue] (\x,-\y) .. controls ({\x+\x/3},{-\y-\x/6}) and ({2*\x-\x/3},{-2*\y+0.05*\x}) .. ({2*\x},{-2*\y});
  }
  \foreach \x in {1.5,3,6} {
    \pgfmathsetmacro{\y}{0.22*\x}
    \draw[thick,red] (\x,\y) .. controls ({\x+0.27*\x},{\y-0.32*\x}) and ({2*\x-0.4*\x},{-2*\y+0.14*\x}) .. ({2*\x},{-2*\y});
    \draw[thick,red] (\x,-\y) .. controls ({\x+0.27*\x},{-\y+0.32*\x}) and ({2*\x-0.4*\x},{2*\y-0.14*\x}) .. ({2*\x},{2*\y});
  }
  \foreach \x in {1.5,3,6} {
    \pgfmathsetmacro{\y}{0.22*\x}
    \foreach \sg in {1,-1} {
      \draw[thick,green!50!black] (\x,\sg*\y) .. controls ({\x+1.0},{\sg*\y+0.05}) .. ({\x+2.2},{\sg*\y+0.13});
      \draw[thick,green!50!black] (\x,\sg*\y) .. controls ({\x+1.0},{\sg*\y-0.05}) .. ({\x+2.2},{\sg*\y-0.13});
    }
  }
  \foreach \x/\k in {1.5/0,3/1,6/2,12/3} {
    \pgfmathsetmacro{\y}{0.22*\x}
    \fill (\x,\y) circle (2pt); \fill (\x,-\y) circle (2pt);
    \node[above left] at (\x,\y) {$P^+_{\k}$};
    \node[below left] at (\x,-\y) {$P^-_{\k}$};
  }
  \node[blue] at (4.3,1.32) {$S_1$};
  \node[red] at (6.27,0.4) {$C_2$};
  \node[green!50!black] at (8.9,1.62) {ends};
\end{tikzpicture}
\caption{Arcs {\textcolor{blue}{$S_k$}}  join $P_k^+$ to $P^+_{k+1}$, arcs {\textcolor{red}{$C_k$}} join
$P_k^+$ to $P^-_{k+1}$, and from every junction two bundles of catenoidal
ends (green) go to infinity. 
Multiplicities are $s_k$, $c_k$, $N$ and, for the disks, $d$.}
\label{f:network}
\end{figure}

\begin{figure}[ht]
\centering
\begin{tikzpicture}[scale=1.55,
    back/.style={densely dashed,very thin},
    notV/.style={line width=0.85pt,line cap=round,dash pattern=on 0pt off 2.3pt}]
  %
  \fill[red!10] (1,0.2) arc[start angle=0,end angle=180,x radius=1,y radius=0.22]
     -- plot[domain=0.8457:1.6734,samples=40,variable=\t]
        ({-0.7249*cosh(\t)},{0.8131-0.7249*\t})
     arc[start angle=180,end angle=360,x radius=2,y radius=0.44]
     -- plot[domain=1.6734:0.8457,samples=40,variable=\t]
        ({0.7249*cosh(\t)},{0.8131-0.7249*\t}) -- cycle;
  \fill[blue!12,even odd rule] (0,0.4) ellipse[x radius=2,y radius=0.44]
                              (0,0.2) ellipse[x radius=1,y radius=0.22];
  \draw[->,gray] (0,-1.35) -- (0,2.45) node[above,black] {$x_3$};
  \draw[gray!70] (1.95,0) -- (-2.8,0) node[left,black] {$\Pi$};
  \foreach \s in {1,-1} {
    \draw[notV,blue] plot[domain=-2.9299:1.9377,samples=110,variable=\t]
      ({\s*0.2822*cosh(\t)},{-0.3469+0.2822*\t});
    \draw[notV,blue] plot[domain=2.6463:2.9299,samples=15,variable=\t]
      ({\s*0.2822*cosh(\t)},{-0.3469+0.2822*\t});
    \draw[notV,red] plot[domain=-1.9701:0.8457,samples=90,variable=\t]
      ({\s*0.7249*cosh(\t)},{0.8131-0.7249*\t});
    \draw[notV,red] plot[domain=1.6734:1.9701,samples=15,variable=\t]
      ({\s*0.7249*cosh(\t)},{0.8131-0.7249*\t});
  }
  \draw[notV,blue] (0,-0.3469) ellipse[x radius=0.2822,y radius=0.062];
  \draw[notV,red] (0,0.8131) ellipse[x radius=0.7249,y radius=0.159];
  \draw[blue,back] (2,0.4) arc[start angle=0,end angle=180,x radius=2,y radius=0.44];
  \draw[blue,thick] (2,0.4) arc[start angle=0,end angle=-180,x radius=2,y radius=0.44];
  \draw[red,back] (2,-0.4) arc[start angle=0,end angle=180,x radius=2,y radius=0.44];
  \draw[red,thick] (2,-0.4) arc[start angle=0,end angle=-180,x radius=2,y radius=0.44];
  \draw[violet,back] (1,0.2) arc[start angle=0,end angle=180,x radius=1,y radius=0.22];
  \draw[violet,very thick] (1,0.2) arc[start angle=0,end angle=-180,x radius=1,y radius=0.22];
  \foreach \s in {1,-1} {
    \draw[blue,very thick] plot[domain=1.9377:2.6463,samples=40,variable=\t]
      ({\s*0.2822*cosh(\t)},{-0.3469+0.2822*\t});
    \draw[red,very thick] plot[domain=0.8457:1.6734,samples=40,variable=\t]
      ({\s*0.7249*cosh(\t)},{0.8131-0.7249*\t});
  }
  \fill (1,0.2) circle (1.0pt);
  \fill (2,0.4) circle (0.85pt);
  \fill (2,-0.4) circle (0.85pt);
  \draw[thin] (1.0,0.17) -- (1.25,-0.42) node[below=-1pt] {\small $P_k^+$};
  \node[anchor=north west] at (2.05,0.39) {\small $P_{k+1}^+$};
  \node[anchor=south west] at (2.05,-0.39) {\small $P_{k+1}^-$};
  \draw[blue,thin] (1.6,0.34) -- (2.25,0.9) node[above right=-2pt] {$S_k$};
  \draw[red,thin] (1.7,-0.26) -- (2.3,-0.88) node[right=-1pt] {$C_k$};
  \draw[blue,thin] (-0.25,-0.37) -- (-0.95,-1.2)
     node[below left=-2pt,align=right] {\small neck of the catenoid\\[-3pt]\small containing $\Rot S_k$};
  \draw[red,thin] (0.71,0.86) -- (1.4,1.45)
     node[right,align=left] {\small neck of the catenoid\\[-3pt]\small containing $\Rot C_k$};
\end{tikzpicture}
\caption{Fitting catenoid pieces together.}
\label{f:two-catenoids}
\end{figure}
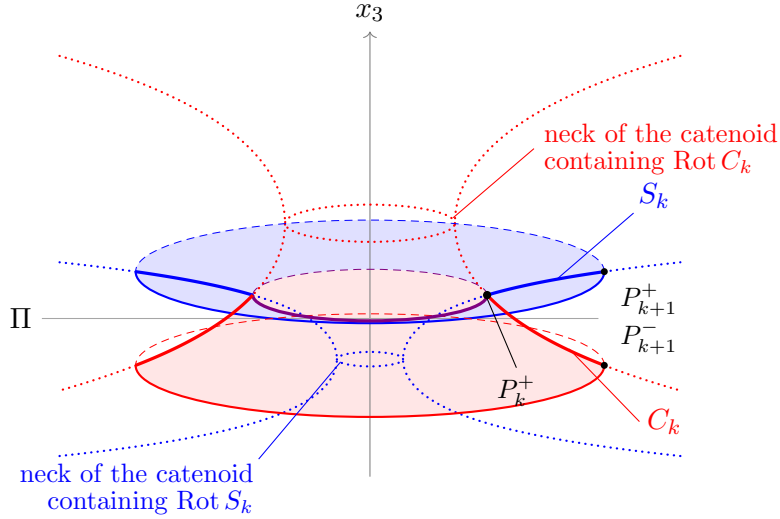

\subsection{Catenary rays}	\label{ss:cat}
Apart from the two disks, every piece of $V$ is a piece of a catenoid with
axis the $x_3$-axis, or a flat annulus. 
The standard catenoid $\{r=\cosh x_3\}$ is minimal \cite[page 9]{CMbook}, and  so
are its dilations and vertical translations $\{r=a\cosh((x_3-b)/a)\}$, with
$a>0$ and $b\in\R$.  Each is the rotation of the catenary
$\{r=a\cosh((h-b)/a)\}\subset\cH$, which has two branches meeting at the
neck $(a,b)$.  We describe each piece of $V$ by its profile, a piece of one
branch, which is determined by its starting point $(r_0,h_0)$ and its
starting slope angle $\alpha$.

\begin{definition}\label{d:catenary}
Let $r_0>0$, $h_0\in\R$, $\alpha\in(-\pi/2,\pi/2)$ and $a=r_0\sin\alpha$.  For
$\rho\ge r_0$ put $h(\rho)=h(\rho;r_0,h_0,\alpha)=h_0$ if $\alpha=0$, and
otherwise
\[
  h(\rho)=b+a\arcosh(\rho/|a|),\qquad b=h_0-a\arcosh(r_0/|a|),
\]
where $\arcosh y=\log(y+\sqrt{y^2-1})$.  We call
$\Gamma(r_0,h_0,\alpha)=\{(\rho,h(\rho)):\rho\ge r_0\}\subset\cH$ a
\emph{catenary ray}, and its part over $[r_0,r_1]$ a \emph{catenary arc}
$\Gamma(r_0,h_0,\alpha;r_1)$.
\end{definition}

Thus, for $\alpha\ne0$, $\Gamma$ is the part of the catenary
$\{r=|a|\cosh((h-b)/a)\}$ where $r\ge r_0$ and $(h-b)/a>0$, that is, of the
branch above the neck if $\alpha>0$ and below it if $\alpha<0$.

On $\R^2$ we write $u(t)=(\cos t,\sin t)$, $e_1=(1,0)$ and
$\sigma(r,h)=(r,-h)$; norms are Euclidean. The \emph{slope angle} of a graph $h=h(\rho)$ is
$\arctan h'(\rho)\in(-\pi/2,\pi/2)$. If $\Rot Y$ has boundary circle
$\Rot\{P\}$, its \emph{inward tangent} at $P$ is the unit vector in $\cH$
tangent to $Y$ at $P$ and pointing into $Y$. 

\vskip1mm
 Lemma~\ref{l:catenary} below collects
what we need about these pieces.  

\begin{lemma}\label{l:catenary}
Let $h=h(\,\cdot\,;r_0,h_0,\alpha)$, $\Gamma=\Gamma(r_0,h_0,\alpha)$ and
$a=r_0\sin\alpha$.
\begin{enumerate}[label=\textup{(\alph*)}]
\item The slope angle $\vartheta(\rho)$ of $\Gamma$ satisfies
$\rho\sin\vartheta(\rho)=r_0\sin\alpha$ for all $\rho\ge r_0$. In particular, the 
 inward tangent is
 $u(\alpha)$  at
$(r_0,h_0)$.
\item If $\alpha>0$, then $h$ is increasing and concave; if $\alpha<0$, then
$h$ is decreasing and convex.
\item For $\rho\ge r_0$ and $x=\rho/r_0$,
\[
  |\sin\alpha|\,\log x\;\le\;\frac{|h(\rho)-h_0|}{r_0}\;\le\;
  |\sin\alpha|\,\big(\log x+\sin^2\alpha\big).
\]
\item For every $r_1>r_0$, the function
$\alpha\mapsto h(r_1;r_0,h_0,\alpha)-h_0$ is continuous on $(-\pi/2,\pi/2)$,
has the sign of $\alpha$, and tends to $\pm r_0\arcosh(r_1/r_0)$ as
$\alpha\to\pm\pi/2$.
\item $h(\lambda\rho;\lambda r_0,\lambda h_0,\alpha)=\lambda\,h(\rho;r_0,h_0,\alpha)$
for $\lambda>0$, and $h(\rho;r_0,-h_0,-\alpha)=-h(\rho;r_0,h_0,\alpha)$.
\end{enumerate}
\end{lemma}

\begin{proof}
Since $\frac{d}{d\rho}\arcosh(\rho/|a|)=(\rho^2-a^2)^{-1/2}$, in all cases
$h'=a(\rho^2-a^2)^{-1/2}$ and $h''=-a\rho\,(\rho^2-a^2)^{-3/2}$.  So
$h'(r_0)=\tan\alpha$ and, as $1+h'^2=\rho^2/(\rho^2-a^2)$,
$\sin\vartheta=h'/\sqrt{1+h'^2}=a/\rho$.  This gives (a) and (b).

For (c) let $\alpha\ne0$.  Integrating $|h'|\ge|a|/\rho$ gives
$|h(\rho)-h_0|\ge|a|\log x$.  Conversely,
$|h(\rho)-h_0|=|a|\log\big((\rho+\sqrt{\rho^2-a^2})/(r_0+\sqrt{r_0^2-a^2})\big)$.
Using $\rho+\sqrt{\rho^2-a^2}\le2\rho$, $r_0+\sqrt{r_0^2-a^2}=r_0(1+\cos\alpha)$
and $\log y\le y-1$,
\[
  \frac{|h(\rho)-h_0|}{|a|}\le\log x+\log\frac{2}{1+\cos\alpha}
  \le\log x+\frac{1-\cos\alpha}{1+\cos\alpha}
  =\log x+\frac{\sin^2\alpha}{(1+\cos\alpha)^2}\le\log x+\sin^2\alpha .
\]
For (d), $F(\alpha)=h(r_1;r_0,h_0,\alpha)-h_0$ equals
$a\big(\arcosh(r_1/|a|)-\arcosh(r_0/|a|)\big)$ for $\alpha\ne0$, and $F(0)=0$.
So $F$ has the sign of $\alpha$ and is continuous away from $0$; it is
continuous at $0$ because $|F(\alpha)|\le r_0|\sin\alpha|\,(\log(r_1/r_0)+1)$
by (c).  As $\alpha\to\pm\pi/2$, $a\to\pm r_0$, so
$F(\alpha)\to\pm r_0\arcosh(r_1/r_0)$ since $\arcosh1=0$.  Finally, (e)
follows from the formula for $h$: replacing $(\rho,r_0,h_0)$ by
$\lambda(\rho,r_0,h_0)$ replaces $(a,b,h)$ by $\lambda(a,b,h)$, and replacing
$(h_0,\alpha)$ by $(-h_0,-\alpha)$ replaces $(a,b,h)$ by $-(a,b,h)$.
\end{proof}

\subsection{Stationary networks of annuli}

The next lemma shows that if the sum of the inward unit conormals, counted with multiplicity, is zero along each boundary circle, then the union of the pieces is stationary.
This is the standard balancing condition; it follows from the first variation formula with boundary  (see, e.g., \cite{Al,CMbook,Si}), and is analogous to  \cite{AllardAlmgren} for one-dimensional stationary networks.

\begin{lemma}\label{l:network}
Let $\{\Sigma_i\}_{i\in I}$ be a countable family of surfaces in $\R^3$, each
of which is $\Rot\Gamma$ for a catenary arc or ray $\Gamma$, or a flat disk
$\{r\le r_i,\ x_3=h_i\}$, and let $m_i$ be positive integers. Assume that
every bounded subset of $\R^3$ meets only finitely many $\Sigma_i$. For a
boundary circle $\Rot\{P\}$ of $\Sigma_i$, let $T_i(P)$ be the inward tangent
of $\Sigma_i$ at $P$ \textup{(}it is $-e_1$ for a disk\textup{)}. If for every
such $P$
\begin{equation}\label{e:balance}
  \sum_{i:\ \Rot\{P\}\subset\partial\Sigma_i} m_i\,T_i(P)=0,
\end{equation}
then $V=\sum_i m_i[\Sigma_i]$ is a stationary integral $2$-varifold in $\R^3$.
\end{lemma}

\begin{proof}
By local finiteness $V$ is an integral varifold; the $\Sigma_i$ may intersect
each other, which plays no role since $V$ is defined as a sum. Let
$X\in C^1_c(\R^3,\R^3)$. Only finitely many $\Sigma_i$ meet $\spt X$, each is
minimal and smooth up to its boundary, and an unbounded $\Sigma_i$ may be
replaced by $\Sigma_i\cap\{r\le\rho_1\}$ with $\rho_1$ so large that $X$
vanishes near $\{r=\rho_1\}$. Hence
\[
  \delta V(X)=\sum_i m_i\int_{\Sigma_i}\operatorname{div}_{\Sigma_i}X
  =\sum_i m_i\int_{\partial\Sigma_i}\langle X,\nu_i\rangle\,d\mathcal H^1,
\]
where $\nu_i$ is the outward unit conormal.  The lemma follows by grouping
 the boundary integrals together for each boundary circle.
 \end{proof}

\subsection{The arcs between junctions}

For the rest of the section we fix $\eta\in(0,1/10)$ and put
$P_k^\pm=2^k(1,\pm\eta)\in\cH$ for $k\ge0$.

\begin{lemma}\label{l:frame}
There are $\alpha_S\in(0,\pi/2)$ and $\alpha_C\in(-\pi/2,0)$ such that, for
every $k\ge0$, the catenary arcs
\[
  S_k=\Gamma(2^k,2^k\eta,\alpha_S;2^{k+1}),\qquad
  C_k=\Gamma(2^k,2^k\eta,\alpha_C;2^{k+1})
\]
start at $P_k^+$ and end at $P^+_{k+1}$ and at $P^-_{k+1}$, respectively.

\noindent {\bf{Initial angles}}. The
numbers $a_S=\sin\alpha_S$ and $a_C=\sin|\alpha_C|$ satisfy
\begin{equation}\label{e:sbounds}
  1.38\,\eta\le a_S\le\frac{\eta}{\log2},\qquad
  3\,\eta\le a_C\le\frac{3\eta}{\log2} .
\end{equation}
{\bf{Final slopes}}.
The arc $S_k$ arrives at $P^+_{k+1}$ with slope angle $\beta_S$, and $C_k$
arrives at $P^-_{k+1}$ with slope angle $-\beta_C$, where
$\beta_S,\beta_C\in(0,\pi/2)$, $\sin\beta_S=a_S/2$ and $\sin\beta_C=a_C/2$.

\noindent
{\bf{Cone condition}}. For $2^k\le\rho\le2^{k+1}$ the height functions $h$ of the arcs
satisfy
\begin{equation}\label{e:arccone}
  \eta\rho\le h(\rho)\le1.1\,\eta\rho\ \text{ on }S_k,\qquad
  |h(\rho)|\le\eta\rho\ \text{ on }C_k .
\end{equation}
\end{lemma}

\begin{proof}
By Lemma~\ref{l:catenary}(e) it suffices to take $k=0$. By
Lemma~\ref{l:catenary}(d) and the intermediate value theorem,
$\alpha\mapsto h(2;1,\eta,\alpha)-\eta$ takes every value in
$(-\arcosh2,\arcosh2)$, which contains $\eta$ and $-3\eta$ since
$3\eta\le0.3<\arcosh2$; and it has the sign of $\alpha$. This gives
$\alpha_S>0$ with $h(2)=2\eta$ and $\alpha_C<0$ with $h(2)=-2\eta$. By
Lemma~\ref{l:catenary}(c) with $x=2$,
\[
  a_S\log2\le\eta\le a_S(\log2+a_S^2),\qquad
  a_C\log2\le3\eta\le a_C(\log2+a_C^2).
\]
The left inequalities are the upper bounds in \eqref{e:sbounds}. They give
$a_S^2\le(0.1/\log2)^2<0.03$ and $a_C^2\le(0.3/\log2)^2<0.2$, and then
the right inequalities give $a_S\ge\eta/(\log2+0.03)\ge1.38\,\eta$ and
$a_C\ge3\eta/(\log2+0.2)\ge3\eta$. The arrival angles follow from
Lemma~\ref{l:catenary}(a) with $\rho=2r_0$.

On $S_0$, $h$ is concave by Lemma~\ref{l:catenary}(b), so it lies above its
chord, which is contained in the line $h=\eta\rho$. By
Lemma~\ref{l:catenary}(c) and \eqref{e:sbounds}, for $x=\rho\in[1,2]$,
\[
  \frac{h(\rho)}{\eta\rho}\le\frac1x\Big(1+\frac{\log x+0.03}{\log2}\Big)
  \le\frac{2}{\log2}\,e^{0.03-1}<1.1,
\]
the middle term being the maximum over $x>0$, attained where
$\log x=1-\log2-0.03$. On $C_0$, $h$ is decreasing and convex, so
$h(\rho)\le h(1)=\eta\le\eta\rho$, and $h$ lies above its tangent line at
$\rho=2$, whose slope is $-\tan\beta_C$. As
$\tan\beta_C\ge\sin\beta_C=a_C/2>\eta$, this gives
$h(\rho)\ge-2\eta+(2-\rho)\tan\beta_C\ge-2\eta+(2-\rho)\eta=-\eta\rho$.
\end{proof}

Since $\alpha_S>0>\alpha_C$,    $\Rot S_k$ lies on the half of a catenoid above its neck, and
$\Rot C_k$ on the half of another catenoid below its neck.
Figure~\ref{f:two-catenoids} shows both annuli leaving the circle
$\Rot\{P_k^+\}$.

For $y,z\in\R^2$ write $[y,z]=y_1z_2-y_2z_1$, so that
$[u(s),u(t)]=\sin(t-s)$, and let $\Delta=\sin(\alpha_S+|\alpha_C|)$, which is
positive as $0<\alpha_S+|\alpha_C|<\pi$.  Every $y\in\R^2$ can then be written
uniquely as $y=s\,u(\alpha_S)+c\,u(\alpha_C)$, and applying
$[\,\cdot\,,u(\alpha_S)]$ and $[u(\alpha_C),\,\cdot\,]$ gives
\begin{equation}\label{e:solve}
  \Delta\,c=[y,u(\alpha_S)],\qquad \Delta\,s=[u(\alpha_C),y] .
\end{equation}
We will use this with
\[
  \lambda_1=\sin(\alpha_S-\beta_S),\quad
  \lambda_2=\sin(\beta_C-\alpha_S),\quad
  \lambda_3=\sin(|\alpha_C|+\beta_S),\quad
  \lambda_4=\sin(|\alpha_C|+\beta_C).
\]

\begin{lemma}\label{l:angles}
The numbers above satisfy:
\begin{enumerate}[label=\textup{(\alph*)}]
\item $\lambda_1,\lambda_4>0$ and $0<\lambda_2<\lambda_3$.
\item $a_S\lambda_2<a_C\lambda_1$.
\item $\cos\beta_j-\cos\alpha_j\ge0.7\,\eta^2$ for $j\in\{S,C\}$.
\end{enumerate}
\end{lemma}

\begin{proof}
As $\sin\beta_j=a_j/2$, and $2a_S<a_C<1/\sqrt2$ by \eqref{e:sbounds} and $\log2>2/3$, we have
$0<\beta_S<\alpha_S<\beta_C<|\alpha_C|<\pi/4$.  This gives (a), since all
the sines have arguments in $(0,\pi/2)$ and $\beta_C-\alpha_S<|\alpha_C|+\beta_S$.
For (b), expanding the sines and using $\sin\beta_j=a_j/2$,
\begin{align}
  a_C\lambda_1-a_S\lambda_2
  &= a_C\, [ \sin(\alpha_S) \cos (\beta_S) -  \cos(\alpha_S) \sin (\beta_S)]
  +a_S \, [ \sin(\alpha_S) \cos ( \beta_C)-  \cos(\alpha_S) \sin \beta_C)] \notag \\
  &= a_S\, [ a_C ( \cos (\beta_S) -  \cos(\alpha_S)) 
  +a_S \,  \cos ( \beta_C)]
  >0 \, ,
\end{align}
where the last inequality used that $0 < \beta_S < \alpha_S < \beta_C < \frac{\pi}{4}$ and $a_C , a_S > 0$.
For (c), $\cos\beta_j-\cos\alpha_j=(\sin^2\alpha_j-\sin^2\beta_j)/(\cos\alpha_j+\cos\beta_j)\ge\frac38a_j^2\ge\frac38(1.38\,\eta)^2>0.7\,\eta^2$
by \eqref{e:sbounds}.
\end{proof}

\subsection{Bundles of ends}

Since the multiplicities of the arcs are integers, the arcs alone can only
balance a junction approximately.  The next two lemmas provide the ends that
absorb the error.

\begin{lemma}\label{l:twoends}
Let $N\in\N$, $0<\psi\le\pi/2$ and $\chi=\cos(\psi/2)$. Every $y\in\R^2$ with
$|y-N(1+\chi)\,e_1|\le N(1-\chi)$ can be written as
$y=N\big(u(t_1)+u(t_2)\big)$ with $t_1,t_2\in[-\psi,\psi]$.
\end{lemma}

\begin{proof}
After dividing through by $N$, we can assume that $N=1$, so $|y-(1+\chi)\,e_1|\le (1-\chi)$. 
The triangle inequality gives that 
 $2\chi\le |y|\le 2$.   Write $y = \bar{r}  \, u (s)$. We claim next that
\begin{align}	\label{e:initi}
	|s|\le\arcsin\frac{1-\chi}{1+\chi}\le\frac\psi2 \, .
\end{align}
To prove the first inequality, observe that
$$
  \bar{r} ^2 + (1+\chi)^2 - 2\, (1+\chi)\,  \bar{r}  \cos s	=|y-(1+\chi)\,e_1^2 |\le (1-\chi)^2 \, ,
$$
and simplifying this and then applying the   inequality $2ab \le a^2 + b^2$  gives
$$
	 \bar{r} ^2 + 4\, \chi \leq 2(1+\chi) \,  \bar{r}  \cos s \leq  \bar{r} ^2 + (1+\chi)^2 \, \cos^2 s \, .
$$
This implies the first inequality in (\ref{e:initi}).   The second inequality follows 
immediately.
Put $\delta=\arccos(|y|/2)\in[0,\psi/2]$ and $t_{1,2}= s \pm\delta$. Then
$|t_j|\le\psi$ and $\big(u(t_1)+u(t_2)\big)=2\cos\delta\;u(s)=y$.
\end{proof}

\begin{lemma}\label{l:endscone}
Let $k\ge0$ and $|\alpha|\le\eta/4$. The catenary ray $\Gamma(2^k,2^k\eta,\alpha)$
satisfies $|h(\rho)|\le1.1\,\eta\rho$ for all $\rho\ge2^k$.
\end{lemma}

\begin{proof}
By Lemma~\ref{l:catenary}(e) we may take $k=0$, so $r_0 = 1$, $h_0 = \eta$, and $\rho\ge1$. By
Lemma~\ref{l:catenary}(c), $\sin|\alpha|\le\eta/4\le1/40$ so we see that
\[
  |h(\rho)-\eta|\le\frac\eta4\Big(\log\rho+\frac1{1600}\Big)\le \frac{\eta}{6400} + \frac{\eta\, \rho}{4\,e}
  \le \eta \, \rho \, \left( \frac{1}{6400} + \frac{1}{4\, e} \right) < \frac{\eta \, \rho}{10} 
\, .
\]
The second inequality used the elementary inequality $\log\rho\le\rho/e$ for $\rho \ge 1$. The lemma follows from this and the triangle inequality since $\eta \le \eta \, \rho$.
\end{proof}

\subsection{The construction and the proof of Theorem \ref{t:varifold}}

Put
\[
  \psi=\eta/4,\qquad N=\lceil 200\,\eta^{-2}\rceil,\qquad
  p=N\big(1+\cos\tfrac\psi2\big)\,e_1,\qquad \mu=1+0.7\,\eta^2 .
\]
By Lemma~\ref{l:twoends}, two bundles of $N$ ends with angles in
$[-\psi,\psi]$ can produce any weighted tangent sum in the disk of radius
$N(1-\cos\frac\psi2)$ about $p$.  We show next that this includes $B_1 (p)$. The Lagrange version of Taylor's Remainder Theorem term gives $\zeta_y \in [-y,y]$  for each  $|y| \le 1/10$ so  that
$1-\cos y=   (\cos y)'' (\zeta_y) \, \frac{y^2}{2} \geq (.45) \, y^2$ and thus
     $$N \bigl(1-\cos\frac\psi2 \bigr) \ge 200\,\eta^{-2} \, \left( .45 \, (\eta/8)^2 \right) 
      >1 \, . $$

\emph{The multiplicities.} For $d\in\N$ we define integers $s_k,c_k$ and
vectors $Y_k$ recursively.  Let $Y_0=de_1$.  Given $Y_k$, let
$s_k^*,c_k^*\in\R$ be the solution of
\begin{equation}\label{e:target}
  s_k^*\,u(\alpha_S)+c_k^*\,u(\alpha_C)=Y_k-p ,
\end{equation}
let $s_k$ and $c_k$ be $s_k^*$ and $c_k^*$ rounded to a nearest integer, and
let $Y_{k+1}=s_k\,u(\beta_S)+c_k\,u(\beta_C)$.  Thus $-Y_k$ is the weighted
sum of the inward tangents at $P_k^+$ of the disk ($k=0$) or of the arcs
arriving there ($k\ge1$), and the arcs leaving $P_k^+$ are chosen to leave the fixed vector $p$, up to rounding, for
the ends; the shift by $p$ is needed because the ends cannot carry a tangent
sum near $0$.  By \eqref{e:solve} and $[u(\alpha_C),e_1]=a_C$,
$[e_1,u(\alpha_S)]=a_S$,
\begin{equation}\label{e:onestep}
\begin{aligned}
  \Delta\,c_0^*&=(d-|p|)\,a_S, &
  \Delta\,s_0^*&=(d-|p|)\,a_C,\\
  \Delta\,c_{k+1}^*&=\lambda_1s_k-\lambda_2c_k-|p|\,a_S, &
  \Delta\,s_{k+1}^*&=\lambda_3s_k+\lambda_4c_k-|p|\,a_C .
\end{aligned}
\end{equation}

\emph{The ends.} Let $\varrho_k=Y_k-s_k\,u(\alpha_S)-c_k\,u(\alpha_C)$; this
is the tangent sum the ends at $P_k^+$ must produce.  By \eqref{e:target},
$\varrho_k=p+(s_k^*-s_k)\,u(\alpha_S)+(c_k^*-c_k)\,u(\alpha_C)$, so
$|\varrho_k-p|\le1<N(1-\cos\frac\psi2)$, and Lemma~\ref{l:twoends} gives
$t_{k,1},t_{k,2}\in[-\psi,\psi]$ with
\begin{equation}\label{e:ends}
  \varrho_k=N\big(u(t_{k,1})+u(t_{k,2})\big).
\end{equation}

By Lemma~\ref{l:angles}(a),(b) we may fix $q$ with
$\max\{a_S/a_C,\,\lambda_1/\lambda_3\}<q<\lambda_1/\lambda_2$.  Let
$H_k=s_k\cos\alpha_S+c_k\cos\alpha_C$ be the horizontal component of the
tangent sum of the arcs leaving $P_k^+$. The next lemma shows that the outgoing horizontal forces grow geometrically, which forces the number of catenoidal pieces to grow as well; this will imply the desired area growth. 

\begin{lemma}\label{l:mult}
If $d$ is large enough, then for all $k\ge0$
\begin{equation}\label{e:growthA}
   H_k\ge\mu^k \qquad\text{and }   1\le c_k\le q \, s_k;
\end{equation}
in particular $s_k\ge1$.
\end{lemma}

\begin{proof}
\emph{Horizontal balance.}  The first component of $\varrho_k$ is at most
$2N$ by \eqref{e:ends}, so $H_k\ge(Y_k)_1-2N$.  Hence $H_0\ge d-2N$, and, if
$s_{k-1},c_{k-1}\ge0$, Lemma~\ref{l:angles}(c) gives
\[
  H_k\ge s_{k-1}\cos\beta_S+c_{k-1}\cos\beta_C-2N
  \ge H_{k-1}+(\mu-1)(s_{k-1}+c_{k-1})-2N\ge\mu H_{k-1}-2N ,
\]
as $s_{k-1}+c_{k-1}\ge H_{k-1}$.
With $L=2N/(\mu-1)$, this says
\begin{equation}\label{e:flux}
  H_k-L\ge\mu\,(H_{k-1}-L).
\end{equation}

\emph{Induction.}  Take $d\ge2N+L+1$, so that $H_0-L\ge1$.  It follows from (\ref{e:flux})
that 
$$
H_k-L\ge\mu^k(H_0-L) \ge \mu^k  , 
$$
giving the first claim in (\ref{e:growthA}). We show next that 
$1\le c_k\le q\,s_k$. 
 For $k=0$, \eqref{e:onestep} shows that $c_0$ and
$s_0$ differ by at most $\frac12$ from $(d-|p|)a_S/\Delta$ and
$(d-|p|)a_C/\Delta$; as $a_S/a_C<q$, we get $1\le c_0\le q\,s_0$ for $d$
large.  Suppose the claim holds for $0, \dots, k$.  Then $H_k\ge H_0\ge d-2N$, and,
since $H_k\le s_k+c_k\le(1+q)\,s_k$, also $s_k\ge(d-2N)/(1+q)$.  By
\eqref{e:onestep}, $0\le c_k\le q\,s_k$, $|p|\le2N$ and rounding,
\[
  \Delta\,c_{k+1}\ge(\lambda_1-q\lambda_2)\,s_k-2Na_S-\tfrac\Delta2,\qquad
  \Delta\,c_{k+1}\le\lambda_1s_k+\tfrac\Delta2,\qquad
  \Delta\,s_{k+1}\ge\lambda_3s_k-2Na_C-\tfrac\Delta2 .
\]
Since $\lambda_1-q\lambda_2>0$ and $q\lambda_3-\lambda_1>0$, these give
$c_{k+1}\ge1$ and $c_{k+1}\le q\,s_{k+1}$ once $s_k$ exceeds a constant
depending only on $\eta$, which holds for $d$ large.  In particular
$s_k,c_k\ge0$, so \eqref{e:flux} holds for $k+1$ and gives
$H_{k+1}-L\ge\mu^{k+1}(H_0-L)$.
\end{proof}

  The stationary varifold $V$ consists of the two initial disks, 
plus the upward and downward catenoid pieces (and their reflections), together with the outgoing catenoid ends. 
The initial disks have multiplicity $d$ from  Lemma~\ref{l:mult} and the other pieces have their multiplicities defined recursively above to achieve balancing (and thus give stationarity). 

\begin{proof}[Proof of Theorem~\ref{t:varifold}]
The construction  immediately gives that $V$ is  an integral varifold satisfying (iii) and (iv). The choices of multiplicities $c_k$ and $s_k$ and the angles of the catenoidal ends give the balancing condition at each singular circle, so Lemma \ref{l:network} gives that $V$ is stationary.

\emph{The cone condition.} By \eqref{e:arccone} and
Lemma~\ref{l:endscone}, every arc and every end, together with its image
under (reflection) $\sigma$, satisfies $|h|\le1.1\,\eta\rho=\varepsilon\rho$ where it
is defined, that is, for $\rho\ge1$. The disks satisfy $|x_3|=\eta \leq \varepsilon \, r + \eta$. This proves (i).

\emph{Growth.}  It suffices to consider the case $R= 2^{k+2}$ for $k\in \N$. 
By (\ref{e:arccone}),
the surfaces $\Rot\tau S_k$ and $\Rot\tau C_k$,
$\tau\in\{\mathrm{id},\sigma\}$, lie in
$\{2^k\le r\le2^{k+1},\ |x_3|\le\varepsilon r\}\subset
B_{(1+\varepsilon^2)^{1/2}R/2}\subset B_R$, and each of them is a graph over
the annulus $\{2^k\le r\le2^{k+1}\}\subset\Pi$, which has area $3\pi4^k$.
Hence, by Lemma~\ref{l:mult}  
\[
  \|V\|(B_R)\ge2(s_k+c_k)\cdot3\pi\,4^k\ge6\pi\,H_k\,4^k\ge6\pi\,(4\mu)^k
  =6\pi\,2^{k(2+\log_2 \mu)}=
  6\pi\, \left( \frac{R}{4} \right)^{2+\log_2 \mu} .
\]
To complete the proof of (ii), it suffices to show that if $0 < \eta < 1/10$, then 
\begin{align}	\label{e:showlog2}
	\log_2 \mu = \log_2 (1 + (.7) \, \eta^2) 
	 >  \frac{4}{5} \, (1.1 \, \eta)^2
	=  \frac{4}{5} \, \varepsilon^2 \, .
\end{align}
We just need to prove the inequality.  This follows from $\frac{4}{5} \, (1.1)^2 \, \log 2 \approx .671 < \frac{70}{101}$ and   $\log  (1 +s)  > \frac{100}{101} \, s$ for $0 < s < .01$ which implies that
$\log_2 ( 1 + (.7) \, \eta^2) \ge \frac{70}{101 \, \log 2} \, \eta^2$.

\end{proof}

\begin{remark}[Smooth examples]\label{r:smooth}
We do not know whether $V$, or a varifold like it, is a limit of smooth
minimal surfaces. A smooth example in the same spirit would require two
further steps. First, the multiplicities would have to be resolved into
distinct sheets. Second, each
junction circle  would
have to be desingularized. Saddle towers \cite{Karcher,PerezTraizet} bent
along circles, as in Kapouleas' desingularization of intersecting catenoids
and planes \cite{Kapouleas}, are natural candidates.  An alternative approach is given by \cite{MM04,MM05}.
\end{remark}

\section{Arbitrary complex dimension and codimension}\label{s:higher-domains}

There are natural higher-dimensional and higher-codimensional versions of the holomorphic results.  Let $1\leq q\leq N$ and $\Pin\subset\C^N$ be a real $2\,q$-plane.    Let $P$ and $P^\perp$
denote the orthogonal projections onto $\Pin$ and $\Pin^\perp$ and define $S_{\varepsilon,K}(\Pin)$ and 
  $C_\varepsilon(\Pin)$ as before.
  
 We will first recall the replacement for the single K\"ahler angle of a two-plane.
It is easy to see that the linear operator $Q \equiv P J|_\Pin : \Pin \to \Pin$ is skew-symmetric.  It follows that if  $v \in \Pin$ is an eigenvector for the
symmetric operator
$Q^T \, Q = - Q^2$, then $Q \, v$ is also an eigenvector with the same eigenvalue and is orthogonal to $v$ (by skew symmetry). 
Thus $-Q^2$ has $q$ eigenvalues, each in $[0,1]$ (since both $P$ and $J$ have operator norm $\le 1$), and each eigenvalue has multiplicity two.  Since the eigenvalues are in $[0,1]$, they can be written as $\cos^2 \theta_i$ for

\[
 0\leq\theta_1\leq\cdots\leq\theta_q\leq \frac{\pi}{2} \, .
\]
The $\theta_i$'s are the \emph{K\"ahler angles} of $\Pin$; note that they might not be distinct. 

\vskip1mm
Below,  $\Sigma$ is a connected $q$ (complex) dimensional manifold and $F:\Sigma\to\C^{N}$ is a proper holomorphic immersion.

\begin{theorem}[Holomorphic thresholds]\label{t:higher-dimensional}
We have three regimes:
\begin{enumerate}
\item[(i)] If $\varepsilon<\tan\frac{\theta_1}{2}$, then no holomorphic $F$ is contained in
$S_{\varepsilon,K}(\Pin)$ for any $K<\infty$.
\item[(ii)] If $\tan\frac{\theta_q}{2} \le \varepsilon<\cot\frac{\theta_q}{2}$, then holomorphic $F$'s  exist, but every one
  is an affine algebraic variety with Euclidean volume growth.
\item[(iii)] If $\varepsilon>\cot\frac{\theta_1}{2}$, 
then there are complete proper holomorphic
immersions of $\D^q$ 
in $C_\varepsilon(\Pin)$ whose volume grows faster than
 $R^{2q}(\log R)^q$.
\end{enumerate}
\end{theorem}

\subsection{Complex linear algebra}

\begin{lemma}	\label{l:spaces}
$\C^N = \R^{2N}$ is the orthogonal sum 
\begin{align}	\label{e:bigde}
\C^N = W \oplus E \oplus JE \oplus \Pi_{\C} \oplus_{ \{ i \, | \, 0 < \theta_i < \pi/2 \} } \, W_i
\end{align}
of the following subspaces:
\begin{enumerate}
\item For each $i$ with $\theta_i \in (0 , \pi/2)$, we get a $4$-dimensional space $W_i$ spanned by
the two-dimensional eigenspace $\Pi_i \subset \Pi$ and $J \, \Pi_i$, $Q: W_i \to \Pi_i$, and  $W_i \perp W_j$ for $i \ne j$.
\item The complex subspace $\Pi_{\C}$ of $\Pi$ corresponding to $\theta_i$'s equal to $0$.
\item The  subspace $E \subset \Pi$ corresponding to $\theta_i$'s equal to $\frac{\pi}{2}$.
\item The  subspace  $J\, E \subset \Pi^{\perp}$.
\item The space $W$ perpendicular to both $\Pi$ and $J \, \Pi$.
\end{enumerate}
Moreover, the spaces $W$, $E \oplus JE$, $\Pi_{\C}$, and each $W_i$ are all preserved by $J$. Finally, 
$P^{\perp}$ is the identity on $W$ and $JE$, vanishes on $\Pi_{\C}$ and $E$, and maps $W_j$ to $W_j$.
\end{lemma}

\begin{proof}
   To get the decomposition, the only things to check are that $Q: W_i \to \Pi_i$ and $W_i \perp W_j$ for $i \ne j$. The first part of this follows since $v$ and $Q\, v$ 
   are orthogonal eigenvectors with the same eigenvalue.  The second part
    follows if we can show that if $v_i \in \Pi_i$ and $v_j \in \Pi_j$ with $i\ne j$, then
\begin{align}	\label{e:fcase}
	\langle v_i , v_j \rangle = 0 = \langle v_i , J \, v_j \rangle  \, .
\end{align}
There are two cases.  First, suppose that $\theta_i \ne \theta_j < \frac{\pi}{2}$. The first claim in (\ref{e:fcase})
 is immediate since eigenvectors for different eigenvalues of $Q^2$ must be orthogonal.  For the second claim, we observe that
\begin{align}
	  \langle v_i , J\, v_j \rangle &=  \langle   v_i , P \, J\, v_j \rangle = \frac{-1}{\cos^2 \theta_i} \,  \langle  Q^2 \,  v_i , Q \, v_j \rangle
	  =  \frac{1}{\cos^2 \theta_i} \,  \langle  Q \,  v_i , Q^2 \, v_j \rangle =  \frac{-\cos^2 \, \theta_j}{\cos^2 \theta_i} \,  \langle  Q \,  v_i , v_j \rangle
	  \notag \\
	  &=   \frac{\cos^2 \, \theta_j}{\cos^2 \theta_i} \,   \langle v_i , J\, v_j \rangle \, , 
\end{align}
giving the second claim in this case.   Finally, suppose that $0 < \theta_{i_1} = \dots \theta_{i_m} < \frac{\pi}{2}$ so that $Q^2$ has a $2m$-dimensional block where it is a (non-zero) constant times the identity.  This time, we need to choose the basis more carefully.  We first choose 
any unit eigenvector $v_{i_1}$ for $Q^2$ and let $\Pi_{i_1}$ be the span of $v_{i_1}$ and $Q \, v_{i_1}$. Now choose $v_{i_2}$ that is perpendicular to  $\Pi_{i_1}$ and observe that so is $Q\, v_{i_2}$ since
\begin{align}
	\langle Q \, v_{i_2} , v_{i_1} \rangle & = - 	\langle v_{i_2} , Q \, v_{i_1} \rangle =  0 \, , \\
	\langle Q \, v_{i_2} , Q \, v_{i_1} \rangle & = 	- \langle v_{i_2} , Q^2 \, v_{i_1} \rangle =  \cos^2 \theta_{i_1} \, 	\langle v_{i_2} , v_{i_1} \rangle = 0 \, .
\end{align}
Now continue to get the desired decomposition. 

The properties of $J$ on the spaces follows immediately from the construction. 
Finally, we establish the claimed properties of $P^{\perp}$. First, since $W$ and $J\, E$ are orthogonal to $\Pin$, $P^{\perp}$ is the identity on both of these.  Both $E$ and $\Pi_{\C}$ are contained in $\Pin$, to $P^{\perp}$ vanishes on these.  Using that $Q$ maps $W_i$ to $\Pi_i$, it follows that $P(W_j) \subset W_j$ and thus
 $P^{\perp}(W_j) \subset W_j$.   
\end{proof}

\begin{lemma}	\label{l:Ldecomp}
If $L$ is a $2q$ dimensional subspace that is preserved by $J$, then we can choose an orthonormal basis of the form
$v_1 , J \, v_1 , \dots , v_q, J \, v_q$ so that
\begin{itemize}
\item $v_i$ has an orthogonal decomposition $v_i = v_i^W + v_i^{\Pi_{\C}} + v_i^E + v_i^{JE} + \sum_{j} v_i^{W_j}$ given by (\ref{e:bigde}).
\end{itemize}
  \end{lemma}

\begin{proof}
We construct the basis inductively.  First choose any unit vector $v_1 \in L$.  Since $L$ is complex, $J\,v_1$ is also in $L$; since $J$ is skew
and orthogonal, $J\, v_1$ is a unit vector perpendicular to $v_1$.  Next, choose  $v_2 \in L$ to be perpendicular to both $v_1$ and $J\, v_1$.
As before $J\, v_2$ is another unit vector in $L$ that is orthogonal to $v_2$. We must verify that $J\, v_2$ is also orthogonal to $v_1$ and $J\, v_1$. Both claims follow from the properties of $J$ 
\begin{align}
	\langle J \, v_2 , v_1 \rangle = - \langle v_2 , J \, v_1 \rangle = 0 \, , \\
	\langle J \, v_2 , J \, v_1 \rangle = \langle v_2 , v_1 \rangle = 0 \, .
\end{align}
Proceeding similarly gives the desired basis.  The orthogonal decomposition follows immediately from Lemma \ref{l:spaces}.
\end{proof}

The following is the analogue of Corollary \ref{c:lines}, with $\theta$ replaced by
$\theta_{q}$.  

\begin{lemma}\label{lem:Lplus}
There are $J$-invariant subspaces $L_{\pm}\subset\C^{N}$ with
\[
   \C^{N}=L_{-}\oplus L_{+}\ \ \text{orthogonally},\qquad
   \dim_{\C}L_{-}=q,\qquad \dim_{\C}L_{+}=N-q,
\]
such that
\begin{align}
   |\Pbot x|^{2}&\ \ge\ \cos^{2}\tfrac{\theta_{q}}{2}\,|x|^{2}
   \qquad\text{for all } x\in L_{+},\label{eq:Lplus}\\
   |\Pbot x|^{2}&\ \le\ \sin^{2}\tfrac{\theta_{q}}{2}\,|x|^{2}
   \qquad\text{for all } x\in L_{-}.\label{eq:Lminus}
\end{align}
Consequently
\begin{enumerate}
\item[(a)] if $\varepsilon<\cot\frac{\theta_{q}}{2}$, then
$L_{+}\cap \Cone{\varepsilon}=\{0\}$;
\item[(b)] if $\varepsilon\ge\tan\frac{\theta_{q}}{2}$, then
$L_{-}\subset \Cone{\varepsilon}$, so the complex $q$-plane $L_{-}$ shows that the
maps in Theorem \ref{t:higher-dimensional}(ii) exist.
\end{enumerate}
\end{lemma}

\begin{proof}
We construct $L_{\pm}$ blockwise using Lemma \ref{l:spaces}.  The space $E$ corresponds to $\theta =\frac{\pi}{2}$ and $E \oplus JE$ contains complex subspaces $L^E_+$ and $L^E_-$ just like the $W_j$'s.  We then define 
\begin{align}
	L_- = \Pi_{\C} \oplus L^E_- \oplus_{j} L^j_- 
	{\text{ and }} L_+ = W \oplus L^E_+ \oplus_{j} L^j_+ \, .
\end{align}
  Lemma \ref{l:spaces} gives the desired properties.
\end{proof}

Let $\pi_-$ and $\pi_+$ denote the orthogonal projections of $\C^{N}$ onto
$L_{-}$ and $L_{+}$, respectively.  Since $L_{\pm}$ are $J$-invariant, $\pi_{\pm}$ are
complex linear, and
\begin{equation}\label{eq:splitnorm}
   |z|^{2}=|\pi_- z|^{2}+|\pi_+ z|^{2}\qquad (z\in\C^{N}).
\end{equation}
We identify $L_{-}\cong\C^{q}$ once and for all.  The next lemma is the analogue
of Lemma \ref{l:zminus}; $\pi_-$ plays the role of the coordinate $z_{-}$.

\begin{lemma}\label{lem:pi}
Suppose $\tan\frac{\theta_{q}}{2}\le \varepsilon<\varepsilon'<\cot\frac{\theta_{q}}{2}$.
Then
\begin{enumerate}
\item There is $c=c(\varepsilon',\theta_{q})\in(0,1]$ so that
\begin{equation}\label{eq:lower}
   c\,|z|\le |\pi_-(z)|\quad\text{and hence}\quad
   |\pi_+(z)|\le c^{-1}|\pi_- (z)|
   \qquad\text{for all } z\in \Cone{\varepsilon'} .
\end{equation}
\item Given $K<\infty$ there is $R_{0}=R_{0}(\varepsilon,\varepsilon',K)$ so that
$\Slab{\varepsilon}{K}\setminus B_{R_{0}}\subset \Cone{\varepsilon'}$.
\end{enumerate}
\end{lemma}

\begin{proof}
Both claims follow as in Lemma \ref{l:zminus} (using
 Lemma \ref{lem:Lplus} in the first).
\end{proof}

\subsection{The proof of (i) and (iii) in Theorem \ref{t:higher-dimensional}}

Arguing as in Lemma \ref{l:minimal-bochner}, 
if $\Sigma$ is minimal and $2\,q$-dimensional, then 
$\rho_{\kappa} (x) = |P\, x|^2 - \kappa \, |P^{\perp} \, x|^2$  satisfies
\begin{align}	\label{e:deltarhok}
	\Delta_{\Sigma} \, \rho_{\kappa} = 4\, q \, \left( 1 - (1+\kappa) \, E_{\Pin} (T\Sigma) \right) \, , 
\end{align}
 where the tilt-excess $ E_{\Pin} (T\Sigma)$ is    the trace of the mapping from $ v \in T\Sigma$ to 
 $\frac{1}{2\,q} \, \left| P^{\perp} \, v \right|^2$.

\begin{proof}[Proof of (i) and (iii) in Theorem \ref{t:higher-dimensional}]
\vskip1mm
{\bf{Part (i)}}: We claim that for each point in the image, we have the lower bound
\begin{align}	\label{e:claimi}
	E_{\Pin} (T\Sigma) \geq \sin^2  \frac{\theta_1}{2} \, .
\end{align}
Since $L = T\Sigma$ is a $2q$-plane that is preserved by $J$ (the map is holomorphic), Lemma 
\ref{l:Ldecomp} gives 
  an orthonormal basis of the form
$v_1 , J \, v_1 , \dots , v_q, J \, v_q$ so that each $v_i$ has the orthogonal decomposition
\begin{align}
	v_i &= v_i^W + v_i^E + v_i^{JE} + \sum_{j} v_i^{W_j}
	  \, ,
\end{align}
where we have that $ v_i^{\Pi_{\C}} =0$ since $\theta_1 > 0$. 
This gives $P^{\perp}$ on $v_i$ and $J\, v_i$ as orthogonal sums 
\begin{align}
	P^{\perp} \, v_i &= v_i^W   + v_i^{JE} + \sum_{j} P^{\perp} \, v_i^{W_j} \, , \\
	P^{\perp} \, (J\, v_i) &= J \, v_i^W   + J \, v_i^{E} + \sum_{j} P^{\perp} \, (J \, v_i^{W_j}) \, .
\end{align}
Squaring these and summing gives that
\begin{align}
	\left| 	P^{\perp} \, v_i \right|^2 + \left| 	P^{\perp} \, J\,  v_i \right|^2 &= 2\, \left|  v_i^W    \right|^2 + \left|  v_i^{E} \right|^2 
	+ \left|  v_i^{JE} \right|^2  + \sum_{j} \left( \left| P^{\perp} \, v_i^{W_j}\right|^2 +  \left| P^{\perp} \, (J \, v_i^{W_j}) \right|^2 \right) \notag \\
	&\geq
	2\, \left|  v_i^W    \right|^2 + \left|  v_i^{E} \right|^2 
	+ \left|  v_i^{JE} \right|^2  + 2\, \sin^2 \frac{\theta_1}{2} \, 
	 \sum_{j}   \left|   v_i^{W_j}\right|^2 	\geq 2\, \sin^2 \frac{\theta_1}{2} \, , 
\end{align}
where the second to last inequality used Lemma \ref{l:mu} applied to each $W_j$ (this applies  since $J$   and 
$P^{\perp}$ preserve $W_j$).  This gives (\ref{e:claimi}).  Summing over $i$ and combining this with (\ref{e:deltarhok}) gives
\begin{align}	\label{e:deltarhok2}
	\Delta_{\Sigma} \, \rho_{\kappa} = 4\, q \, \left( 1 - (1+\kappa) \, E_{\Pin} (T\Sigma) \right)
	\leq 2\, q \, \left( 1 -  (1+\kappa) \, \sin^2 \frac{\theta_1}{2} \right) 
	 \, . 
\end{align}
If we choose $\varepsilon_0 \in (\varepsilon, \tan \frac{\theta_1}{2})$ and set $\rho =\rho_{\varepsilon_0^{-2}}$, then it follows that
$\Delta_{\Sigma} \, \rho \leq 0$.   Claim (i) now follows from the maximum principle as in the proof of Proposition \ref{p:tangent-obstruction}.

\vskip1mm
{\bf{Part (iii)}}:
The hypothesis of \textup{(iii)} implies that $\theta_1>0$, so  $\bar{W}=\Pin+J\Pin=W_1\oplus\cdots\oplus W_q$ is a complex $2q$-plane in $\C^N$, where $W_j$ corresponds to $\theta_j$ and thus is either one of the $W_j$'s from 
Lemma \ref{l:spaces} (when $\theta_j < \frac{\pi}{2}$) or part of $E \oplus JE$ (when $\theta_j = \frac{\pi}{2}$).
 Since
$\theta_1\leq\theta_j$, we have
\[
 \varepsilon>\cot\frac{\theta_1}{2}\geq\cot\frac{\theta_j}{2}
 \qquad (1\leq j\leq q).
\]
Part (iii) in Theorem \ref{t:angleversion}, applied in $W_j\cong\C^2$, gives a complete proper holomorphic
immersion $G_j:\D\to W_j$ with image in $C_\varepsilon(\Pin_j)$ and
\[
 \frac{A_{G_j}(R)}{R^2\log R}\longrightarrow\infty.
\]
Define
\[
 G(\zeta_1,\ldots,\zeta_q)
   =G_1(\zeta_1)+\cdots+G_q(\zeta_q).
\]
The orthogonal decompositions of $\Pin$ and $\Pin^\perp$ show that
$G(\D^q)\subset C_\varepsilon(\Pin)$.  The map is proper as a map into the closed complex
subspace $\bar{W}\subset\C^N$, and its product metric is complete.  Moreover, the volume of  $G^{-1}(B_R)$
in the induced (product) metric is at least
\[
  \prod_{j=1}^q A_{G_j}\!\left(\frac{R}{\sqrt q}\right),
\]
which gives the asserted volume growth.   
\end{proof}

\subsection{The proof of  (ii) in Theorem \ref{t:higher-dimensional}}

The argument is modeled on the case of curves with some differences in higher dimensions.


From now on we fix $\varepsilon<\cot\frac{\theta_{q}}{2}$ and $K<\infty$, choose
$\varepsilon'$ with $\max\{\varepsilon,\tan\frac{\theta_{q}}{2}\} < \varepsilon'
<\cot\frac{\theta_{q}}{2}$ (replacing $\varepsilon$ by
$\tan\frac{\theta_{q}}{2}$ if $\varepsilon<\tan\frac{\theta_{q}}{2}$, which only
enlarges the cone), and let $c$ and $R_{0}$ be given by Lemma \ref{lem:pi}.  Define holomorphic maps
\begin{equation}\label{eq:hg}
   h=\pi_- \circ F:\Sigma \to\C^{q},\qquad g=\pi_+\circ F:\Sigma \to L_{+} \, .
\end{equation}

\begin{corollary}\label{cor:proper}
If $F(\Sigma)\subset\Slab{\varepsilon}{K}$, then
\begin{equation}\label{eq:sandwich}
   |h|\ \le\ |F|\ \le\ c^{-1}|h|+R_{0}\qquad\text{on } \Sigma .
\end{equation}
In particular 
$h:\Sigma\to\C^{q}$ is a proper holomorphic map. Moreover, $h$ is a finite branched covering with a
well defined degree $D<\infty$.
\end{corollary}

\begin{proof}
The left inequality in \eqref{eq:sandwich} is immediate from \eqref{eq:splitnorm}.   
 If $|F(p)|\ge R_{0}$, then $F(p)\in \Cone{\varepsilon'}$ by Lemma
\ref{lem:pi}(2) and hence $|F(p)|\le c^{-1}|h(p)|$ by \eqref{eq:lower}; otherwise
$|F(p)|<R_{0}$.  This gives the second inequality, also implying properness of $h$.

The  image of $F$ that projects to a fixed point $w$  in $L_-$ is discrete because it is an analytic subset of 
  $(w+L_{+}) \subset  \Cone{\varepsilon'}\cup
B_{R_{0}}$ and a compact analytic
subset of a complex linear space must be finite.   Since $F$ is a proper immersion, it follows
that $h$ has discrete fibers.
Finally, 
a map between
equidimensional complex manifolds that is proper and has discrete fibers is a
finite branched covering of constant degree by the covering lemma in chapter $7$ of \cite{GR}. (If $F$ has
generic rank $q$ but is not injective on a divisor, $D$ is the generic fiber
count, which is all that is used below.)
\end{proof}

We use $d=\partial+\bar\partial$ and $d^{c}=\frac{i}{2}(\bar\partial-\partial)$, so that
$\ddc=i\,\partial\bar\partial$ and $du\wedge d^{c}u=i\,\partial u\wedge\bar\partial u$
for real $u$.  We write $\omega_{2q}=\pi^{q}/q!$ for the volume of the unit ball in
$\R^{2q}$, $\{|h|<r\}$ for $\{p\in\Sigma:|h(p)|<r\}$, and
$V_F(R)=\Vol\bigl(F^{-1}(B_R)\bigr)$; volumes on $\Sigma$ are computed in the induced
metric and hence counted with multiplicity.

Introduce the semipositive $(1,1)$-forms on $\C^{N}$ and their pullbacks,
\begin{equation}\label{eq:forms}
   \alpha=\tfrac12 \ddc |\pi_- z|^{2},\qquad
   \gamma=\tfrac12\ddc|\pi_+ z|^{2},\qquad \beta=\alpha+\gamma=\tfrac12\ddc|z|^{2},
\end{equation}
\[
   \alpha_{F}=F^{*}\alpha=\tfrac12\ddc |h|^{2},\qquad
   \gamma_{F}=F^{*}\gamma=\tfrac12\ddc|g|^{2},\qquad
   \beta_{F}=\alpha_{F}+\gamma_{F} .
\]
Here $\beta$ is the standard K\"ahler form of $\C^{N}$, and $\alpha$ restricts to the
standard K\"ahler form of $L_-\cong\C^q$.  By Wirtinger's theorem \cite{W}, for every
Borel set $U\subset \Sigma$,
\begin{equation}\label{eq:wirtinger}
   \Vol (U)=\frac{1}{q!}\int_{U}\beta_{F}^{\,q}.
\end{equation}
Pullbacks of semipositive $(1,1)$-forms under holomorphic maps are semipositive, and
wedge products of semipositive $(1,1)$-forms are positive.  Hence
$\alpha_{F}^{\,k}\wedge\gamma_{F}^{\,m}\ge 0$ as a measure whenever $k+m=q$, and in
particular $\alpha_F^{\,q}\le\beta_F^{\,q}$, since
$\beta_F^{\,q}=\sum_{m=0}^{q}\binom{q}{m}\alpha_F^{\,q-m}\wedge\gamma_F^{\,m}$.

The substitute for the constancy \eqref{e:e36} of the flux of $\log|h|$ is the following
degree identity.  By Corollary \ref{cor:proper}, $h$ is a proper holomorphic map of
degree $D$ onto $\C^{q}$, and $\alpha_F^{\,q}/q!$ is the pullback under $h$ of the
Euclidean volume form of $L_-\cong\C^q$.  The area formula therefore gives, for every
$r>0$,
\begin{equation}\label{eq:degree}
   \frac{1}{q!}\int_{\{|h|<r\}}\alpha_{F}^{\,q}
   = D\,\omega_{2q}\,r^{2q}.
\end{equation}
   If $R_1\ge R_0$, then \eqref{eq:sandwich} and $c\le1$ give
$|F|<c^{-1}R_1+R_0\le 2 \, c^{-1} \, R_1$ on $\{|h|<R_1\}$, so
$\{|h|<R_1\}\subset F^{-1}(B_{ 2 \, c^{-1} \, R_1})$.  Combining this with \eqref{eq:degree},
$\alpha_{F}^{\,q}\le \beta_{F}^{\,q}$ and \eqref{eq:wirtinger} bounds the degree by the
volume at a single scale:
\begin{equation}\label{eq:degreebound}
   D\ \le\ \frac{1}{\omega_{2q}R_{1}^{2q}}\,\Vol\{|h|<R_{1}\}
     \ \le\ \frac{1}{\omega_{2q}R_{1}^{2q}}\,V_{F}( 2 \, c^{-1} \, R_{1})
     \qquad\text{for } R_1\ge R_0 .
\end{equation}

Fix once and for all a cutoff $\chi\in C^{\infty}([0,\infty),[0,1])$ with
$\chi\equiv1$ on $[0,1]$ and $\chi\equiv 0$ on $[2,\infty)$, and set
\begin{equation}\label{eq:cutoff}
   \eta_{r}=\chi\!\left(\frac{|h|^{2}}{r^{2}}\right),\qquad
   C_{\chi}=2\|\chi'\|_{\infty}+4\|\chi''\|_{\infty} .
\end{equation}
Then $\eta_r=1$ on $\{|h|<r\}$, and $\eta_r$ is supported in
$\{|h|\le\sqrt2\,r\}$, which is compact by Corollary \ref{cor:proper}.  Since $\chi'$
and $\chi''$ vanish outside $(1,2)$, the forms $d\eta_r$ and $\ddc\eta_r$ vanish outside
\[
   U_r=\{r<|h|<\sqrt2\,r\}.
\]
The following lemma is the heart of the matter; it plays the role that
Lemma \ref{l:twofcts} played for $q=1$.

\begin{lemma}[Transfer of $\gamma_{F}$ into $\alpha_{F}$]\label{lem:transfer}
If $F(\Sigma)\subset\Slab{\varepsilon}{K}$, 
$k+m=q$ with $m\ge 1$ and $r\ge R_{0}$, then
\begin{equation}\label{eq:transfer}
   \int_{\{|h|<r\}}\alpha_{F}^{\,k}\wedge\gamma_{F}^{\,m}
   \ \le\  3C_\chi \,c^{-2}\int_{\{|h|<\sqrt2\,r\}}
           \alpha_{F}^{\,k+1}\wedge\gamma_{F}^{\,m-1}.
\end{equation}
\end{lemma}

\begin{proof}
All forms below are pullbacks under $F$ of smooth forms on $\C^{N}$, so they are
smooth on $\Sigma$, and integrals over $\Sigma$ count multiplicity automatically.  Put
$\Theta=\alpha_{F}^{\,k}\wedge\gamma_{F}^{\,m-1}$, a closed semipositive
$(q-1,q-1)$-form.  If $\phi$ is a compactly supported smooth function and $u$ a smooth
real function on $\Sigma$, then
$d\phi\wedge d^cu\wedge\Theta=du\wedge d^c\phi\wedge\Theta$, since only the
$(1,1)$-parts of $d\phi\wedge d^cu$ and $du\wedge d^c\phi$ contribute and these are equal.  Stokes'
theorem therefore gives
$\int_\Sigma\phi\,\ddc u\wedge\Theta=\int_\Sigma u\,\ddc\phi\wedge\Theta$.  Applying
this with $\phi=\eta_r$ and $u=|g|^2$, and using $\gamma_{F}=\frac12\ddc|g|^{2}$,
\begin{equation}\label{eq:ibp}
   \int_{\Sigma}\eta_{r}\,\alpha_{F}^{\,k}\wedge\gamma_{F}^{\,m}
   =\frac12\int_{U_r}|g|^{2}\ \ddc\eta_{r}\wedge\Theta .
\end{equation}
Two estimates finish the proof.

\smallskip\noindent
\emph{The cone bound.}  This is where the hypothesis
$F(\Sigma)\subset\Slab{\varepsilon}{K}$ enters, through \eqref{eq:sandwich}.  On $U_r$
we have $|h|<\sqrt2\,r$, so \eqref{eq:sandwich} and $R_{0}\le r\le c^{-1}r$ give
\begin{equation}\label{eq:conebound}
   |g|^{2}\le |F|^{2}\le \big(c^{-1}\sqrt2\,r+R_{0}\big)^{2}
   \le \big(1+\sqrt 2\big)^{2}c^{-2}r^{2}\le 6\,c^{-2}r^{2} .
\end{equation}

\smallskip\noindent
\emph{The cutoff bound.}  Writing $t=|h|^{2}/r^{2}$,
\[
   \ddc \eta_{r}=\chi'(t)\,\frac{\ddc|h|^{2}}{r^{2}}
   +\chi''(t)\,\frac{d|h|^{2}\wedge d^{c}|h|^{2}}{r^{4}} .
\]
Since $h$ is holomorphic, $\partial|h|^{2}=\sum_{j}\bar h_{j}\,dh_{j}$ and
$\ddc|h|^{2}=i\sum_{j}dh_{j}\wedge d\bar h_{j}$.  For every $(1,0)$-vector $\xi$, the
Cauchy--Schwarz inequality
$\big|\sum_{j}\bar h_{j}\,dh_{j}(\xi)\big|^{2}\le |h|^{2}\sum_{j}|dh_{j}(\xi)|^{2}$
therefore gives the inequality of semipositive $(1,1)$-forms
\begin{equation}\label{eq:CS}
   d|h|^{2}\wedge d^{c}|h|^{2}=i\,\partial |h|^{2}\wedge\bar\partial|h|^{2}
   \ \le\ |h|^{2}\ \ddc |h|^{2}=2|h|^{2}\alpha_{F} .
\end{equation}
Since $|h|^{2}\le 2r^{2}$ on $U_r$ and $\ddc\eta_r=0$ outside $U_r$, \eqref{eq:CS}
yields, on all of $\Sigma$,
\begin{equation}\label{eq:cutoffbound}
   \pm\,\ddc\eta_{r}\ \le\ \frac{C_{\chi}}{r^{2}}\,\alpha_{F} .
\end{equation}

\smallskip\noindent
\emph{Conclusion.}  Wedging the inequality \eqref{eq:cutoffbound} of $(1,1)$-forms with
the semipositive form $\Theta$ preserves it, and $|g|^2\ge0$.  Since $\eta_r\ge0$,
$\eta_r=1$ on $\{|h|<r\}$ and $\alpha_F^{\,k}\wedge\gamma_F^{\,m}\ge0$,
\eqref{eq:ibp}, \eqref{eq:conebound} and \eqref{eq:cutoffbound} give
\begin{align*}
   \int_{\{|h|<r\}}\alpha_{F}^{\,k}\wedge\gamma_{F}^{\,m}
   &\le\int_{\Sigma}\eta_{r}\,\alpha_{F}^{\,k}\wedge\gamma_{F}^{\,m}
   \le\frac12\cdot 6\,c^{-2}r^{2}\cdot\frac{C_\chi}{r^{2}}
      \int_{U_r}\alpha_{F}^{\,k+1}\wedge\gamma_{F}^{\,m-1}\\
   &\le 3\,C_\chi\,c^{-2}\int_{\{|h|<\sqrt2\,r\}}\alpha_{F}^{\,k+1}\wedge\gamma_{F}^{\,m-1}.
\end{align*}
\end{proof}

\begin{proposition}[Euclidean volume growth]\label{prop:vol}
Suppose that $F(\Sigma)\subset\Slab{\varepsilon}{K}$.  With $c$ and $R_{0}$ as in
Lemma \ref{lem:pi},     and
\begin{equation}\label{eq:C1}
   C_{1}=\omega_{2q}\big(1+2^{q} \, 3 \, C_{\chi} \,  c^{-2}\big)^{q},
\end{equation}
we have for all $R\ge R_{1}\ge R_{0}$
\begin{equation}\label{eq:volgrowth}
   V_{F}(R)\ \le\ \Vol\{|h|<R\}\ \le\ C_{1}\,D\,R^{2q}
   \ \le\ \frac{C_{1}}{\omega_{2q}}\ \frac{R^{2q}}{R_{1}^{2q}}\ V_{F}( 2 \, c^{-1} \, R_{1}).
\end{equation}
\end{proposition}

\begin{proof}
Set $A_{k,m}(r)=\int_{\{|h|<r\}}\alpha_{F}^{\,k}\wedge\gamma_{F}^{\,m}\ge 0$ for
$k+m=q$.  Let $r\ge R_0$.  Applying Lemma \ref{lem:transfer} $m$ times, at the radii
$r,\sqrt2\,r,\dots,2^{(m-1)/2}r$, all of which are at least $R_0$, and then the degree
identity \eqref{eq:degree} at the radius $2^{m/2}r$ gives
\begin{equation}\label{eq:iterate}
   A_{q-m,m}(r)\ \le\ \big( 3 \, C_{\chi} \, c^{-2}\big)^{m}A_{q,0}\big(2^{m/2}r\big)
   \ =\ \big(  3 \, C_{\chi} \, c^{-2}\big)^{m}\,q!\,D\,\omega_{2q}\,2^{mq}\,r^{2q}
   \qquad(0\le m\le q).
\end{equation}
Expanding $\beta_{F}^{\,q}$ as above and using \eqref{eq:wirtinger} and
\eqref{eq:iterate},
\begin{equation}\label{eq:sum}
   \Vol\{|h|<r\}=\frac{1}{q!}\sum_{m=0}^{q}\binom{q}{m}A_{q-m,m}(r)
   \le D\,\omega_{2q}\,r^{2q}\sum_{m=0}^{q}\binom{q}{m}
   \big(2^{q} \, 3 \, C_{\chi} \,  c^{-2}\big)^{m}
   = C_{1}\,D\,r^{2q} .
\end{equation}
The first inequality in \eqref{eq:volgrowth} holds because
$F^{-1}(B_R)\subset\{|h|<R\}$ by \eqref{eq:sandwich}, the second is \eqref{eq:sum}
with $r=R$, and the third is the degree bound \eqref{eq:degreebound}.
\end{proof}

\begin{proof}[Proof of (ii) in Theorem \ref{t:higher-dimensional}]
If $\tan\frac{\theta_{q}}{2}\le\varepsilon<\cot\frac{\theta_{q}}{2}$, then
Lemma \ref{lem:Lplus}(b) gives $L_{-}\subset\Cone{\varepsilon}$, so the inclusion of the
complex $q$-plane $L_-$ is such a map.  Now let $F$ be any such map.  By
Proposition \ref{prop:vol} with $R_1=R_0$, we have $V_{F}(R)\le C_1D\,R^{2q}$ for all
$R\ge R_0$, so $F$ has Euclidean volume growth.  Since $F$ is proper, the proper mapping
theorem (Theorem 8.8 in \cite{Dem}) shows that $F(\Sigma)$ is a closed analytic subset
of $\C^N$ of pure dimension $q$, and its $2q$-dimensional volume in $B_R$ is at most
$V_F(R)$.  By the theorem of Bishop and Stoll \cite{Bishop,S}, $F(\Sigma)$ is an affine
algebraic variety.
\end{proof}

\end{document}